\documentclass[a4paper,10pt]{scrartcl}
\usepackage[margin=1in]{geometry}
\usepackage[T1]{fontenc}
\usepackage[utf8]{inputenc}
\usepackage[english]{babel}
\usepackage{cite}
\usepackage{paralist}
\usepackage{amsmath,amssymb,amstext}
\usepackage{esint}
\usepackage{nicefrac}
\usepackage{hyperref}
\usepackage{pict2e}
\usepackage{graphicx}
\usepackage{mathdots}
\usepackage{tikz}
\usepackage[]{algorithm2e}
\usepackage{siunitx}
\usepackage{pgfplots}
\usepackage{dsfont}
\usepackage{authblk}

\usepackage{todonotes}

\pgfplotsset{compat=1.13}

\DeclareFontFamily{U}{mathx}{\hyphenchar\font45}
\DeclareFontShape{U}{mathx}{m}{n}{
	<5> <6> <7> <8> <9> <10>
	<10.95> <12> <14.4> <17.28> <20.74> <24.88>
	mathx10
}{}
\DeclareSymbolFont{mathx}{U}{mathx}{m}{n}
\DeclareMathSymbol{\bigtimes}{1}{mathx}{"91}

\usepackage[standard,thmmarks,amsmath,hyperref]{ntheorem}
\numberwithin{equation}{section}

\renewtheorem{theorem}[z]{Theorem}
\renewtheorem{lemma}[z]{Lemma}
\renewtheorem{corollary}[z]{Corollary}
\renewtheorem{definition}[z]{Definition}
\renewtheorem{example}[z]{Example}
\renewtheorem{remark}[z]{Remark}

\theoremstyle{nonumberbreak}

\DeclareMathOperator{\N}{\mathbb{N}}

\DeclareMathOperator{\R}{\mathbb{R}}

\DeclareMathOperator{\id}{\mathrm{d}\!}
\DeclareMathOperator{\eps}{\varepsilon}
\DeclareMathOperator{\diag}{diag}

\DeclareMathOperator{\rank}{rank}

\DeclareMathOperator{\epsM}{\eps_{\text{M}}}

\newcommand\numberthis{\addtocounter{equation}{1}\tag{\theequation}}

\title{\LARGE Rank Bounds for Approximating Gaussian Densities in the Tensor-Train Format}

\author[1]{\large Paul~B.~Rohrbach}
\affil[1]{Department of Applied Mathematics and Theoretical Physics, University of Cambridge, Wilberforce Road, Cambridge CB3 0WA, United Kingdom (\href{mailto:pbr28@cam.ac.uk}{pbr28@cam.ac.uk})}
\author[2]{Sergey~Dolgov}
\affil[2]{Department of Mathematical Sciences, University of Bath, Claverton Down, Bath BA2 7AY, United Kingdom (\href{mailto:S.Dolgov@bath.ac.uk}{S.Dolgov@bath.ac.uk})}
\author[3]{Lars~Grasedyck}
\affil[3]{Institut f\"ur Geometrie und Praktische Mathematik, RWTH Aachen, Templergraben 55, 52056 Aachen, Germany (\href{mailto:lgr@igpm.rwth-aachen.de}{lgr@igpm.rwth-aachen.de})}
\author[4,2]{Robert~Scheichl}
\affil[4]{Institute for Applied Mathematics, Heidelberg University, Im
  Neuenheimer Feld 205, 69120 Heidelberg, Germany
  (\href{mailto:R.Scheichl@uni-heidelberg.de}{R.Scheichl@uni-heidelberg.de})}
\date{\large \today\vspace{-2ex}}

\begin{document}
\maketitle

\begin{abstract}
    Low-rank tensor approximations have shown great potential for
      uncertainty quantification in high dimensions, for example, to
      build  surrogate models that can be used to speed up large-scale
      inference problems (Eigel et al., \emph{Inverse Problems} \textbf{34}, 2018; Dolgov
      et al., \emph{Statistics \& Computing} \textbf{30}, 2020). The feasibility and
      efficiency of such approaches depends critically on the rank
      that is necessary to represent or approximate the underlying
      distribution.  In this paper, a-priori rank bounds for
      approximations in the functional tensor-train representation for
      the case of Gaussian models are
      developed. It is shown that under suitable conditions on the
      precision matrix, the Gaussian density can be approximated to
      high accuracy without suffering from an exponential growth of
      complexity as the dimension increases. These results provide a
      rigorous justification of the suitability and the limitations of
      low-rank tensor methods in a simple but important model
      case. Numerical experiments confirm that the rank bounds capture
      the qualitative behavior of the rank structure when varying the
      parameters of the precision matrix and the accuracy of the
      approximation. Finally, the practical relevance of the theoretical
      results is demonstrated in the context of a Bayesian filtering problem.
\end{abstract}

\section{Introduction}
Inference problems for high-dimensional random variables appear commonly in scientific computing.
In the field of uncertainty quantification, for example, the behavior of a system of interest can be modeled by the pushforward of a random coefficient field by a physical model (often given by a partial differential equation), see e.g. \cite{lordetal2014}, or by the posterior distribution of the corresponding Bayesian inverse problem of estimating parameters based on measured (noisy) data \cite{stuart2010inverse}.
In typical applications, the underlying stochastic parameter domain is infinite dimensional and has to be discretized (e.g.\ by a truncated Karhunen-Loève expansion) with tens to thousands of stochastic parameters resulting in a complicated high-dimensional random variable that needs to be investigated.
A great deal of effort has been spent on developing numerical methods for such problems.

One specific example of a Bayesian inverse problem is filtering~\cite{J70},
where one aims to assimilate noisy time series data into a probabilistic dynamical state space model.
Under the assumptions of linear dynamics and Gaussian noise, the resulting Gaussian filtering distribution can be computed exactly via the Kalman filter. In the general nonlinear case, a Bayes-optimal filter
requires the solution of a high-dimensional Fokker-Planck equation to find the exact (non-Gaussian) filtering distribution, requiring efficient general-purpose numerical approximation techniques in high dimensions.
In moderate dimensions, direct piecewise polynomial approximation methods,
e.g., based on sparse grid approaches \cite{bungartzgriebel2004}, can be used to soften the exponential growth of the cost with respect to dimension.
However, these methods do not fundamentally escape the curse of dimensionality.

Sampling based methods, in particular Monte Carlo or Markov chain Monte Carlo methods, are in principle suitable for inference and uncertainty quantification (UQ) in very high dimensions \cite{liu2004,stuart2010inverse,lordetal2014}.
However, the slow convergence rates and the potentially long decorrelation times of the Markov chains for general, high-dimensional, nonlinear inverse problems can make these methods very expensive, in particular when the forward model is complicated and costly to evaluate. In the linear Gaussian case, direct samplers based on factorizations of the posterior covariance matrix are available \cite{rue2001, bardsley2012, aune2013, fox-norton2016} and  polynomial-accelerated Markov chains can be used \cite{fox-parker2017}.

For non-Gaussian distributions, several promising, recently developed approaches to improve the performance of UQ and inference techniques (either direct or sampling-based) for general nonlinear problems in high dimensions are based on low-rank tensor surrogates.
Originating from renormalization group procedures in computational chemistry \cite{white1992}, such surrogates can be used, e.g., to represent (or approximate) the high dimensional distribution in a parametric format, with a complexity that grows only polynomially with dimension and thus allows to break the curse of dimensionality. The low-rank tensor format can be seen as a generalization of the matrix singular value decomposition to higher dimensions.
There is no unique way to do this and there are a range of competing formats.
Here, we focus on the commonly used Tensor-Train decomposition \cite{oseledets2011tensor, oseledets2009breaking}.
For reviews on low-rank tensor formats see \cite{hackbuschtensor,grasedyck2013literature,khoromskij2015tensor}.

There are many ways to utilize low-rank parametrizations in UQ.
For example, they can be used directly to approximate the forward map \cite{khoromskij2011tensor,dolgov2015polynomial,ballani-grasedyck2015, eigel2017adaptive, dolgov2019hybrid} or the posterior distribution \cite{Eigel-bayes-2018} in variational approaches, or to speed up sampling from the exact distribution \cite{dolgov2018approximation,cui-dolgov2020}.
Their efficiency hinges on the size of the ranks that are necessary to represent or to approximate the underlying high dimensional function sufficiently accurately.
Low-rank tensor approaches appear to be working well in practice in the applications considered in the papers above. However, little is known about theoretical guarantees on the required ranks.
Exact rank bounds have only been shown for some specific function classes in other applications (not UQ-related) so far \cite{khoromskij2010dmrg,grasedyck2010polynomial,dk-qtt-tucker-2013,kva-anidiff-2013}. In particular, so far there exists no theory for probability density estimation via low-rank techniques.

In this paper, we present first results on a-priori bounds for the rank of tensor approximations in the case of high dimensional Gaussian random variables.
Due to their ubiquity in statistics as the limiting distribution of (sufficiently regular) averages of random variables, this is a canonical starting point for developing a theory of low-rank methods in UQ.
Our considerations are based on \emph{locality} of correlations between different variables, e.g.\ bandedness of the covariance matrix, which was also used to reduce the computational complexity of MCMC~\cite{morzfeld2017localization}.
When the covariance matrix of the Gaussian is diagonal, the corresponding density reduces to a product of one-dimensional functions which is clearly of rank $1$.
We show that we can preserve low-rank approximability when we are sufficiently close to this setting.
We quantify this by looking at the singular spectrum of the subdiagonal blocks of the precision matrix (inverse covariance matrix): we can approximate the densities with low rank if there are only few singular values (Theorem \ref{thm:lowrankapprox}) or the values are decaying fast (Theorem \ref{thm:expdecayapprox}).
In those cases, we can prove a poly-logarithmic/polynomial growth rate of the ranks with respect to the inverse approximation accuracy $1/\eps$ and with respect to the dimension $d$. This result breaks down as the rank of the subdiagonal block matrices increases or, correspondingly, as the decay rate of the singular values decreases.

While the range of direct applications of these theoretical results for Gaussian distributions is limited, the technical difficulties are already substantial and provide a natural first stepping stone for further analysis of the algorithms presented for example in \cite{Eigel-bayes-2018, dolgov2018approximation,cui-dolgov2020}. Nevertheless, our theoretical results can be used directly, for example to estimate ranks of TT approximations of the solution to the Fokker-Planck equation in Bayesian filtering. We demonstrate this in Section  \ref{sec:filtering_example} for a general nonlinear dynamical system. At each observation time $t_\ell$, in addition to solving the Fokker-Planck equation numerically, we also compute the linearised posterior covariance $C(t_\ell)$ cheaply via the Kalman filter. Via our new theoretical results, the singular values of the subdiagonal blocks of the precision matrix $C(t_\ell)^{-1}$ will provide rigorous bounds on the ranks necessary to approximate the corresponding Gaussian density function to a specified accuracy. Provided the dynamics is weakly nonlinear or the frequency of observations is high, we can expect that the TT ranks of the (non-Gaussian) Bayes-optimal filtering distribution of the original nonlinear dynamical system will be very similar. This conjecture is motivated by holomorphy of many commonly employed likelihood functions~\cite{schwab2020bayes} and the fact that polynomial correction functions (e.g. via a Taylor expansion) can be represented in the TT format with moderate TT ranks~\cite{khoromskij2010dmrg}.
A detailed analysis of all possible sub-Gaussian functions is however beyond the scope of this paper.

The paper is organized as follows: In Section \ref{sec:tt} we introduce the Tensor-Train format and some preliminary material. 
The main results are presented in Section \ref{sec:rankbounds}.
In Section \ref{sec:numerics}, we provide numerical examples that confirm qualitatively our theoretical results, and then illustrate in Section \ref{sec:filtering_example} how they can be used within a nonlinear Bayesian filtering problem.
Finally, we summarize our results in Section~\ref{sec:conclusions}.

\section{Low-Rank Tensor Decompositions}\label{sec:tt}
In this section, we give a brief overview of the relevant tensor decomposition methods. 
From an abstract point of view, the goal of these methods is to represent a high dimensional function
\[
	f: \R^d \to \R
\]
as well as its discrete evaluation $T = f(\widehat{Q}) \in \R^{n_1 \times \dots \times n_d}$ on a tensor grid
\[
	\widehat{Q} = \bigtimes_{i=1}^{d} \{\xi^{(i)}_1, \dots, \xi^{(i)}_{n_i}\}
\]
with $n_i \in \N$ grid points on each axis.
First, we look at the discrete case of the tensor $T$.
It is obvious that computing and storing all elements of a tensor $T = f(\widehat{Q})$ directly is prohibitively expensive for anything but very small dimensions.
In order to be able to work with $T$, we need an efficient representation whose complexity does not increase exponentially in its dimension.
In two dimensions, a low-rank representation and optimal approximations are easily computable using the singular value decomposition (SVD).
There is no clear way to generalize such a decomposition to higher dimensions and there are various formats that manage to transfer some of the properties of the SVD to higher dimensions \cite{hackbuschtensor, grasedyck2013literature}.
In this paper, we will use the Tensor-Train (TT) format \cite{oseledets2011tensor, oseledets2009breaking} and the corresponding functional Tensor-Train (FTT) format \cite{oseledets2013constructive,bigoni2016spectral,gorodetsky2019continuous} as the low parametric representations.
However, the results presented in this paper can be easily generalized to other subspace based tensor formats, e.g. the hierarchical Tucker format \cite{grasedyck2010hierarchical}.

In this section, we briefly define the format and state the properties needed for the following proofs.
For an introduction to the format and further details, see \cite{oseledets2011tensor} and the references in this section.
\begin{definition}[Discrete Tensor-Train Representation] \label{def:tttensor}
	Let $r_0, \dots, r_{d} \in \N$, and $T \in \R^{n_1 \times \cdots \times n_d}$ be a $d$-dimensional tensor. 
	The tuple of $3$-tensors
	\[
		(G_1, \dots, G_d) \;\; \text{ with } \;\; G_j \in \R^{r_{j-1} \times n_j \times r_j}
		\;\; \text{ where } \;\; r_0 = r_d = 1
	\]
	is called a Tensor-Train (TT) representation of $T$ if
	\begin{equation}\label{eqn:evaluatingTT}
		T(i_1, \dots, i_d) = G_1(i_1) \cdots G_d(i_d)
			\;\; \text{where} \;\; G_j(i_j) := \left[ G_j(l, i_j, r) \right]_{l,r=1}^{r_{j-1},r_j}.	
	\end{equation}
\end{definition}
The tensor $T$ is represented or approximated by a sequence of three-dimensional tensors $G_j$ called the \textit{TT cores} (see Figure \ref{Img:TTFormat}).
We call $r = (r_0, \dots, r_d)$ the TT ranks of the decomposition.
For $n \sim n_j$ and bounded ranks $r \sim r_j$, the complexity of a TT tensor grows of the order $\mathcal{O}(dnr^2)$.
Each entry $T(i_1, \dots, i_d)$ is computed by a product of $d$ matrices which are given by the $i_j$-th slice of the $3$ dimensional TT core $G_j(i_j)$, see Equation (\ref{eqn:evaluatingTT}).
Since $r_0=1$, the first slice $G_1(i_1) \in \R^{r_1}$ is a one-dimensional vector and thus the evaluation of an entry of $T$ reduces to the computation of $d$ matrix-vector products which can be computed in $\mathcal{O}(dr^2)$ operations.
In physics literature, this format is also known as the Matrix Product States (MPS) representations of the tensor $T$ \cite{schollwock2005density}.
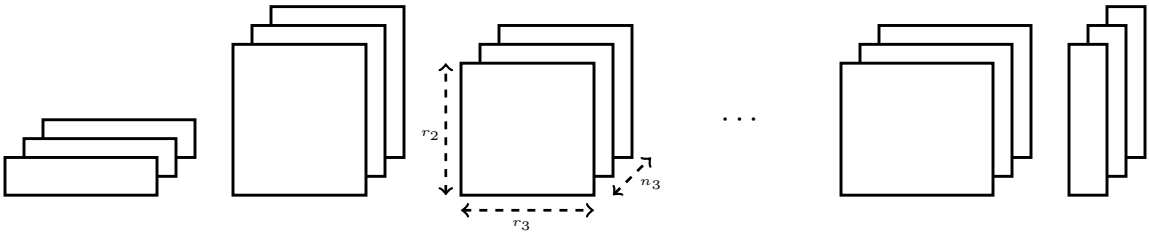
\begin{figure}[ht!]
	\centering


\begin{tikzpicture}
	\begin{scope}
		\draw[fill=white, line width=1.1] (0.5, 0.5) rectangle (2.5, 1);
 		\draw[fill=white, line width=1.1] (0.25, 0.25) rectangle (2.25, 0.75);
		\draw[fill=white, line width=1.1] (0, 0) rectangle (2, 0.5);
	\end{scope}
	
	\begin{scope}[shift={(3, 0)}]
		\draw[fill=white, line width=1.1] (0.5, 0.5) rectangle (2.25, 2.5);
 		\draw[fill=white, line width=1.1] (0.25, 0.25) rectangle (2, 2.25);
		\draw[fill=white, line width=1.1] (0, 0) rectangle (1.75, 2);
	\end{scope}
	
	\begin{scope}[shift={(6, 0)}]
		\draw[fill=white,line width=1.1] (0.5, 0.5) rectangle (2.25, 2.25);
 		\draw[fill=white, line width=1.1] (0.25, 0.25) rectangle (2, 2);
		\draw[fill=white, line width=1.1] (0, 0) rectangle (1.75, 1.75);
		
		\draw[dashed, <->, line width=1.02] (2, 0) -- (2.5, 0.5);
		\draw (2.5, 0.15) node {\tiny $n_3$};
		
		\draw[dashed, <->, line width=1.02] (0, -0.2) -- (1.75, -0.2);
		\draw (0.8, -0.4) node {\tiny $r_3$};
		
		\draw[dashed, <->, line width=1.02] (-0.2, 0) -- (-0.2, 1.75);
		\draw (-0.4, 0.8) node {\tiny $r_2$};
	\end{scope}
	
	\begin{scope}[shift={(9, 0)}]
		\draw (0.7, 1) node {\large $\cdots$};
	\end{scope}
	
	\begin{scope}[shift={(11, 0)}]
		\draw[fill=white,line width=1.1] (0.5, 0.5) rectangle (2.5, 2.25);
 		\draw[fill=white, line width=1.1] (0.25, 0.25) rectangle (2.25, 2);
		\draw[fill=white, line width=1.1] (0, 0) rectangle (2, 1.75);
	\end{scope}
	
	\begin{scope}[shift={(14, 0)}]
		\draw[fill=white,line width=1.1] (0.5, 0.5) rectangle (1, 2.5);
 		\draw[fill=white, line width=1.1] (0.25, 0.25) rectangle (0.75, 2.25);
		\draw[fill=white, line width=1.1] (0, 0) rectangle (0.5, 2);
	\end{scope}

\end{tikzpicture}

	\caption{Visualization of the cores of a TT tensor.}\label{Img:TTFormat}
\end{figure}

The Tensor-Train rank structure of $T$ is determined by the properties of the linear subspaces spanned by the matricizations of the tensor.
\begin{definition}[Matricization]\label{def:matricisation}
	Let $T \in \R^{n_1 \times \dots \times n_d}$ and $l \in \{1, \dots, d-1\}$. We call the matrix that results from reshaping $T$ to
	\begin{equation}
	T^{(l)}(i_1, \dots, i_l; i_{l+1}, \dots, i_d) \in \R^{n_1 \cdots n_l \times n_{l+1} \dots n_d}
	\end{equation}
	the $l$-matricization of $T$.
\end{definition}
The minimal TT ranks needed to exactly represent a tensor in the TT format are given by the matrix ranks of these matricizations.
\begin{theorem}\label{thm:ttranks}
	Let $T \in \R^{n_1 \times \dots \times n_d}$. For any exact TT decomposition $(G_1, \dots, G_d)$ of $T$
	\begin{equation}
		T(i_1, \dots, i_d) = G(i_1) \cdots G(i_d),
	\end{equation}
	the $l$-th TT rank $r_l$ fulfills 
	\begin{equation} \label{eqn:lowerrankbound}
		r_l \geq \rank T^{(l)}.
	\end{equation}
	Furthermore, there exists a TT decomposition of $T$ with ranks $r_l = \rank T^{(l)}$.
\end{theorem}
\begin{proof}
	Let $(G_1, \dots, G_d)$ be a TT decomposition of $T$. We define
	\[
		G_{\leq l}(i_1, \dots, i_l) = G_1(i_1) \cdots G_{l}(i_l) \in \R^{1 \times r_l}, \;\;
		G_{>l}(i_{l+1}, \dots, i_d) = G_{l+1}(i_{l+1}) \cdots G_d(i_d) \in \R^{r_l \times 1}.
	\]
	Then $T(i_1, \dots, i_d) = G_{\leq l}(i_1, \dots, i_{l}) \cdot G_{> l}(i_{l+1}, \dots, i_d)$ is a skeleton decomposition of $T^{(l)}$, implying $r_l \geq \rank T^{(l)}$.
	It is shown in \cite[Theorem 2.1]{oseledets2011tensor} that the TT-SVD algorithm explained below constructs a TT decomposition where the rank bound (\ref{eqn:lowerrankbound}) is achieved with equality.
\end{proof}
Given the full tensor $T$, we can compute a TT decomposition by successively computing singular value decompositions.
We sketch the procedure: we start with the first matricization
\begin{equation}
	T^{(1)} = U_1 \Sigma_1 V_1^\intercal.
\end{equation}
Reshaping the first matrix of the decomposition $U_1 \in \R^{n_1 \times r_1}$ to the $3$-tensor $G_1 \in \R^{r_0 \times n_1 \times r_1}$ (recall that $r_0=1$) gives us the first TT core of the decomposition. The remaining matrix $(\Sigma_1 V_1^\intercal) \in \R^{r_1 \times n_2 \cdots  n_d}$ is reshaped to $\widehat{(\Sigma_1 V_1^\intercal)} \in \R^{r_1 n_2 \times n_3 \cdots  n_d}$.
Again, we compute the SVD
\begin{equation}
	\widehat{(\Sigma_1 V_1^\intercal)} = U_2 \Sigma_2 V_2^\intercal
\end{equation}
and reshape $U_2 \in \R^{r_1 n_1 \times r_2}$ to $G_2 \in \R^{r_1 \times n_2 \times r_2}$ to get the second TT core. 
This process continues with $(\Sigma_2 V_2^\intercal)$ and so forth.
As the final step, we set the last tensor core to be $G_d = {\Sigma_{d-1} V_{d-1}^\intercal} \in \R^{r_{d-1}\times n_d \times r_{d}}$.
The resulting cores $(G_1, \dots, G_d)$ form a TT decomposition of $T$ with minimal ranks \cite[Theorem 2.1]{oseledets2011tensor}.
Since this algorithm relies on subsequent singular value decompositions it is known as \textit{TT-SVD}.

Since random matrices have full rank almost surely \cite{feng2007rank}, it is an immediate consequence of Theorem \ref{thm:ttranks} that the exact representations of almost all random tensors $T$ in $\R^{n_1 \times \cdots \times n_d}$ have full rank in each matricization and their complexity scales exponentially in $d$.
In practice, we are interested in tensors that arise from systems with structure that could lead to bounded ranks.
Nevertheless, any noise in the entries would make an exact representation impossible even in a case where the underlying tensor is of low rank.
It is therefore important to be able to compute approximations and quantify the occurring error.
For the Tensor-Train format (and other subspace based formats like the hierarchical Tucker format \cite{grasedyck2010hierarchical}), best low-rank approximations are guaranteed to exist and the corresponding errors can be easily bounded by again looking at properties of the induced linear subspaces.
\begin{theorem} \label{thm:ttapproximation}
	Let $T \in \R^{n_1 \times \cdots \times n_d}$ be a tensor and $r = (r_0, \dots, r_{d})$ with $r_0 = r_d = 1$ be a rank bound. 
	Then, there exists a best approximation $\widetilde{T}$ of $T$ in the Frobenius norm with TT ranks bounded by $r$ whose error is bounded by
	\begin{equation}\label{eqn:ttapprox}
		\|T - \widetilde{T}\|_{F}  = \inf_{\rank(G) \leq r} \|T - G\|_F \leq \left(\sum_{l=1}^{d-1} \inf_{\rank A \leq r_l} \| T^{(l)} - A \|_{F}^2 \right)^{1/2}.
	\end{equation}
\end{theorem}
For a proof, see \cite[Theorem 2.2, Corollary 2.4]{oseledets2011tensor}.
One of the key features of the TT format is the right-hand side bound in Equation (\ref{eqn:ttapprox}).
This enables us to bound the tensor approximation error by only looking at approximations of its matricizations which is a much simpler linear algebra problem.
This observation will be the key ingredient that we use to prove the rank bounds in Section \ref{sec:rankbounds}.

In practice, computing exact best approximations is generally intractable, but by subsequently approximating the subspace of each matricization in the TT-SVD algorithm using a truncated singular value decomposition, we can compute an approximation that achieves the bound (\ref{eqn:ttapprox}) \cite[Theorem 2.2]{oseledets2011tensor}.
Since the the overall approximation error in the TT format is always larger than the error of a single matricization in the Frobenius norm, such an approximation is quasi-optimal with an additional factor of $\sqrt{d-1}$ \cite[Corollary 2.4]{oseledets2011tensor}.
The truncated TT-SVD procedure can be implemented efficiently for tensors already given in the TT format \cite[Algorithm 1]{oseledets2011tensor}.

To extend this format to a functional setting, we will follow the same recipe as outlined above, however applied on appropriate function spaces.
In each step, we need to find an equivalent of the operation we used in the discrete case.
For details of this construction, see \cite{bigoni2016spectral}.
We start by defining the functional Tensor-Train (FTT) decomposition of a function $f \in L^2(\R^d)$.
\begin{definition}
	Let $f \in L^2(\R^d)$ and
    \[
        \gamma_i: \N \times \R \times \N \to \R \;\; \text{ for } \;\; i=2, \dots, d-1,
	\]
	\[
        \gamma_1: \{1\} \times \R \times \N \to \R, \;\; \gamma_d: \N \times \R \times \{1\} \to \R
    \]
    be measurable functions.
	We call $(\gamma_1, \dots, \gamma_d)$ a functional Tensor-Train decomposition of $f$ if 
	\begin{equation}\label{eqn:truncfttdecomp}
		f(x_1, \dots, x_d) = \sum_{\alpha_1, \dots, \alpha_{d-1}= 1}^{\infty} \gamma_1(\alpha_0, x_1, \alpha_1) \cdots \gamma_d(\alpha_{d-1}, x_d, \alpha_d), \;\; \alpha_0 = \alpha_d = 1,
	\end{equation}
	in $L^2(\R^d)$.
\end{definition}
Instead of a tuple of $3$-tensors, we have decomposed the $d$-dimensional function $f$ into a tuple of $3$-dimensional functions (the FTT cores) with one real and two countable indices.
The ranks for the functional decomposition can be infinite.
The rank $r_i$ is finite if $\gamma_i(\cdot, \cdot, \alpha_i) \equiv \gamma_{i+1}(\alpha_i, \cdot, \cdot) \equiv 0$ for $\alpha_i > r_i$ (i.e.\ we only need sum over index $\alpha_i$ from $0$ to $r_i$ in Equation (\ref{eqn:truncfttdecomp})).

The properties of this format are again determined by the structure of the functional matricization
\begin{equation}
	f^{(l)}(x_1, \dots, x_l; x_{l+1}, \dots, x_d): \R^{l} \times \R^{d-l} \to \R.
\end{equation}
To analyze this, we need an analogue of the singular value decomposition on $L^2$ function spaces.
This is given by the Schmidt decomposition \cite[Theorem 4]{vsimvsa1992bestl}.
\begin{theorem}[Schmidt decomposition] \label{thm:schmidt}
	Let $Q_1 \subset \R^l, Q_2 \subset \R^{d-l}$ be open sets and $Q = Q_1 \times Q_2$. Let $\mu = \mu_{Q_1} \otimes \mu_{Q_2}$ be a $\sigma$-finite measure and $f \in L^2_\mu(Q)$.
    Then, there exist complete orthogonal systems
	\begin{equation}
		\{\gamma_i\}_{i=1}^{\infty} \subset L^2_{\mu_{Q_1}}(Q_1),
			 \;\; \{\phi_i\}_{i=1}^{\infty} \subset L^2_{\mu_{Q_2}}(Q_2)
	\end{equation}
	and a non-decreasing sequence $(\lambda_i)_{i=1}^{\infty} \subset \R$ such that $f$ can be expanded as the $L^2_\mu(Q)$ converging series
	\begin{equation}
		f = \sum_{i=1}^{\infty} \sqrt{\lambda_i} \gamma_i \phi_i.
	\end{equation}
	Furthermore, the partial sums $\sum_{i=1}^{k} \sqrt{\lambda_i} \gamma_i \phi_i$ are the orthogonal projections of $f$ and yield the best low-rank approximations
	\begin{equation}
		\| f - \sum_{i=1}^{k} \sqrt{\lambda_i} \gamma_i \phi_i \|_{L^2_{\mu}(Q)}
			= \sqrt{\sum_{i=k+1}^{\infty} \lambda_i} 
			= \inf_{\substack{g \in L^2_{\mu}(Q_1 \times Q_2)\\ \rank g = k}} \|f - g \|_{L^2_\mu(Q)}.
	\end{equation}
\end{theorem}
Starting with any function $f \in L^2(\R^d)$, we can explicitly construct an FTT decomposition of $f$ by following the steps of the TT-SVD algorithm and iteratively applying the Schmidt decomposition to the functional matricizations. 
The details of this FTT-SVD algorithm are given in \cite[Section 4.1]{bigoni2016spectral}.

The remaining ingredient that we require for our results is the ability to bound the error of a finite-rank FTT approximation to a given function $f \in L^2(\R^d)$.
As before, this is determined by the ability to approximate the matricizations.
\begin{theorem}\label{thm:fttapproximation}
	Let $f \in L^2(\R^d)$ and $r = (r_0, \dots, r_d)$. Then, there exists an approximation $\widetilde{f}$ of $f$ with FTT-rank $r$ whose error is bounded by
	\begin{equation}
		\|\widetilde{f} - f\|_{L^2(\R^d)} \leq \left( \sum_{l=1}^{d-1} 
			\inf_{\substack{g \in L^2(Q_l^1 \times Q_l^2)\\ \rank g = r_l}} 
			\|f - g\|_{L^2(\R^d)}^2 \right)^{1/2},
	\end{equation}
	where
	\[
		Q_l^1 := \R^{l}\;\; \text{and} \;\; Q_l^2 := \R^{d-l}.
	\]
\end{theorem}
This result can be shown by following the steps of \cite[Proposition 9]{bigoni2016spectral} and using Theorem \ref{thm:schmidt}.
Similar to the discrete case, we can construct an approximation that achieves this bound by truncating the representations of the matricizations of $f$ during the FTT-SVD algorithm to $r$.

Let $f \in L^2(\R^d) \cap C(\R^d)$ be a function with FTT rank $r = (r_0, \dots, r_{d})$.
If we compute the evaluation $f(\widehat{Q})$ of $f$ on the tensor grid $\widehat{Q}$, it can be written as
\begin{equation}
	f(\widehat{Q})(i_1, \dots, i_d) 
		= \sum_{\alpha_1, \dots, \alpha_{d-1}}^{r} \gamma_1(1, \xi^{(1)}_{i_1}, \alpha_1) 
		\gamma_2(\alpha_1, \xi^{(2)}_{i_2}, \alpha_2) \cdots \gamma_d(\alpha_{d-1}, \xi^{(d)}_{i_d}, 1).
\end{equation}
Therefore, the discrete TT ranks of $f(\widehat{Q})$ are always bounded by the functional TT ranks of $f$ independent of the grid $\widehat Q$ it is evaluated on.
Of course, the discrete rank might be lower in practice, especially for coarse grids.

\subsection{Computing FTT Approximations}
\label{sec:computingftt}

For the numerical tests in Section \ref{sec:numerics}, we need to be
able to compute approximations of functions in the FTT format. We restrict ourselves to smooth functions $f: Q \to \R$ on a bounded domain
\begin{equation}
    Q := [-a, a]^d.
\end{equation}
The reason for the restriction to a bounded domain is purely technical; it is a requirement for our numerical approximation scheme of $f$.
The influence of restricting a function $f \in L^2(\R^d)$ to $Q$ on the FTT ranks of approximations can be made arbitrarily small by choosing $a$ appropriately.
Therefore, being able to approximate functions on a bounded domain is sufficient for testing their rank structure.
For details, see the beginning of Section \ref{sec:numerics} where we investigate the dependence of the ranks of Gaussian densities on the size of the domain.

We use multivariate polynomial interpolation to approximate $f$.
For simplicity, we will assume the same number of $n$ nodes for each dimension.
Let $\widehat{Q} := \{\xi_1, \dots, \xi_n\}^d \in [-a, a]^{n^d}$ be a discrete tensor grid in $Q$ and $T := f(\widehat{Q})$.
We consider $T$ as the sample points of a multivariate Lagrange interpolation polynomial
\begin{equation}
	p_T(x) = \sum_{i_1, \dots, i_d = 1}^{n} T(i_1, \dots, i_d) l^{(1)}_{i_1}(x_1) \cdots l^{(d)}_{i_d}(x_d)
\end{equation}
with $l^{(i)}_j$ the corresponding $j$-th one dimensional Lagrange basis functions for the node basis $\{\xi_1, \dots, \xi_n\}$.
If
\begin{equation}
	T(i_1, \dots, i_d) = G_1(i_1) \cdots G_d(i_d)
\end{equation}
is a TT tensor of rank $r := (r_0, \dots r_{d})$, the FTT rank of the corresponding interpolation polynomial $p_T$ is bounded by $r$ since it can be written as
\begin{equation}\label{eqn:ttpoly}
	p_T(x) = \sum_{(\alpha_1, \dots, \alpha_{d-1})=1}^{r} \left( \sum_{i_1 = 1}^{n} G_1(i_1, \alpha_1) l^{(1)}_{i_1}(x_1) \right)
		\cdots \left( \sum_{i_d = 1}^{n} G_d(\alpha_{d-1}, i_d) l^{(d)}_{i_d}(x_d) \right).
\end{equation}
To avoid the instability problems generally associated with polynomial interpolation, we choose the nodes $\{\xi_1, \dots, \xi_n\}$ of the Gauss quadrature rule on $[-a, a]$.
This choice has the additional benefit that it enables us to approximate integrals of $f$, and in particular also its $L^2$ norm, with high accuracy.
Since we use the unscaled Lebesgue product measure on $Q$, the nodes are given by the transformed roots of the $n$-th Legendre polynomial.
The resulting interpolation polynomial $p_T$ in (\ref{eqn:ttpoly}) can be efficiently evaluated at any point $x \in \R^d$ using  \cite[Procedure 3]{bigoni2016spectral}.

Assume that we are given the exact evaluation tensors $T_n$ for growing node sizes $n$.
Then, the interpolation polynomial $p_n := p_{T_n}$ converges fast to the target function $f$.
\begin{theorem}\label{thm:polyapprox}
	For any $\nu \in \N$, there exists a constant $C = C(\nu)$ such that the Legendre interpolation polynomial $p_n$ converges with rate $\nu$ against $f$
	\begin{equation}
		\|f - p_n\|_{L^2(Q)} \leq C(\nu) n^{-\nu} |f|_{Q, \nu},
	\end{equation}
	where $|\cdot|_{Q, \nu}$ denotes the Sobolev semi-norm of the $\nu$-th derivative on $Q$.
\end{theorem}
For a proof of this theorem, see \cite[Proposition 6]{bigoni2016spectral}.
To evaluate integrals of $f$, we apply the corresponding Gaussian quadrature rule.
Let $w = (w_1, \dots, w_n)$ be the quadrature weights corresponding to the Gauss-Legendre nodes $\{\xi_1, \dots, \xi_n\}$.
Then, we have
\begin{equation}
	\int_Q f(x) \id x \approx 
    \int_{Q} p_T(x) \id x = \sum_{i_1, \dots, i_d = 1}^{n} (T \circ \mathcal{W})(i_1, \dots, i_d) \;\; \text{with} \;\; \mathcal{W} := \bigotimes_{i=1}^d w
\end{equation}
where $\circ$ denotes the Hadamard product of tensors.
Using this, we can easily determine the $L^2$ norm of $p_T$ by computing the Frobenius norm of the suitably scaled evaluation tensor
\begin{equation}
		\|p_T\|_{L^2(Q)} = \| T \circ \sqrt{\mathcal{W}} \|_F.
\end{equation}
The Hadamard product and the Frobenius norm are easily computable for tensors in the TT format \cite{oseledets2011tensor}.

Using interpolation, the problem of computing a
functional approximation is discretized to computing the tensor $T$.
Of course, the discrete tensor is still far too large to be evaluated
directly and a structure-adapted algorithm for its approximation has
to be employed.
Here, we use the TT-cross algorithm in \cite{oseledets2010tt}, namely its rank-adaptive version implemented in the \texttt{rect\_cross} class of the \texttt{ttpy} Python package \cite{oseledets2018ttpy}.
Using this, we can compute high-accuracy approximations of the
discrete  evaluation tensors of $f$.

There are two contributions to the error of $p_T$: the interpolation error and the error of the TT-cross approximation of the node tensor.
For the exact evaluation $f(\widehat{Q})$, the error of the polynomial $p_{f(\widehat{Q})}$ only depends on the grid $\widehat{Q}$ and, using Theorem \ref{thm:polyapprox}, we can make the relative interpolation error
\begin{equation}
	\delta_{\text{int}} := \|f - p_{f(\widehat{Q})}\|_{L^2(Q)} /  \|f\|_{L^2(Q)}
\end{equation}
arbitrary small.
The approximation of $T \approx f(\widehat{Q})$ introduces an additional error
\begin{equation}
	\delta_{\text{appr}} := \|p_{f(\widehat{Q})} - p_T\|_{L^2(Q)} /  \|f\|_{L^2(Q)} = \|(f(\widehat{Q}) - T) \circ \sqrt{\mathcal{W}} \|_F /  \|f\|_{L^2(Q)}.
\end{equation}
The overall error is bounded by
\begin{equation}
	\|f - p_T\|_{L^2(Q)}/ \|f\|_{L^2(Q)} \leq \delta_{\text{int}} + \delta_{\text{appr}} =: \delta.
\end{equation}
To test the FTT rank structure of $f$, we compute low-rank approximations by truncating the ranks of the TT tensor $T \circ \sqrt{\mathcal{W}}$ to relative accuracy $\eps$ giving us an approximation $\tilde{T}$ of $T$.
Consequently, we have
\begin{align}
	\|p_{\tilde{T}} - f\|_{L^2(Q)} &\leq \|p_{\tilde{T}} - p_{T}\|_{L^2(Q)} + \|p_T - f\|_{L^2(Q)} \\
		&\leq \|(T - \tilde{T}) \circ \sqrt{\mathcal{W}}\|_F + \|p_T - f\|_{L^2(Q)} \\
		&\leq (\delta + \eps) \|f\|_{L^2(Q)}.
\end{align}
For $\eps \gg \delta$, we can ignore the error contribution due to our approximation scheme $\delta$ and we thus obtain a functional approximation with relative error $\eps$ and the FTT ranks of $p_{\tilde{T}}$.

For a given computed approximation, we need to judge its relative accuracy $\delta$.
The TT-cross procedure generates an error estimate of the completed tensor.
To check that we have chosen a sufficiently dense grid, we draw random points on $Q$ from the target density and compute an importance sampling estimate of the relative interpolation error on $Q$.
If this error agrees with the error estimate of the TT-cross procedure, we accept the completed tensor and use it to investigate the rank structure up to an accuracy $\eps$ of one order of magnitude less than $\delta$, see Section \ref{sec:numerics}.

\section{Bounds on Low-Rank Approximability}
\label{sec:rankbounds}

In this section, we present our main results.
We define two conditions on the precision matrix of a Gaussian density under which we can bound the rank growth for approximations in the functional Tensor-Train format as a function of the dimension and the relative accuracy.
The setup is this: given a symmetric positive definite precision matrix $\Gamma \in \R^{d \times d}$, we consider the density of the Gaussian random variable $X \sim \mathcal{N}(0, \Gamma^{-1})$
\begin{equation}\label{eqn:def_f_gamma}
	f_\Gamma: \R^d \to \R, \;\; x \mapsto e^{-\frac{1}{2} x^\intercal \Gamma x}
\end{equation}
where we have dropped the normalization factor to simplify the notation (this does not influence the ranks).
If the precision matrix $\Gamma = \diag(\gamma_1, \dots, \gamma_d)$ is diagonal, $f_\Gamma$ immediately factorizes to
\begin{equation}\label{eqn:diagdensity}
    f_\Gamma(x_1, \dots, x_d) = \prod_{i=1}^{d} e^{-\frac{1}{2} \gamma_i x_i^2}
\end{equation}
which is a rank $1$ function.
On the other hand, we do not typically see low ranks in approximations when we choose the precision matrix randomly.
A reasonable expectation would be that we can approximate the density easily when the precision matrix is ``close'' to being diagonal.
In this paper, we present one possible notion of this for which we can prove explicit rank bounds and which is general enough to be useful in applications, see Section \ref{sec:filtering_example}.

We relate the approximability of $f_\Gamma$ to the structure of the subdiagonal blocks of the precision matrix $\Gamma$.
For $k = 1, \dots, d-1$, we can split $\Gamma$ into submatrices
\begin{equation}
   	\Gamma = \begin{bmatrix}
   				\Gamma_{1,k} & A_k^\intercal \\
   				A_k & \Gamma_{2,k}
  			 \end{bmatrix},
         \;\; A_k \in \R^{(d-k) \times k}.
\end{equation}
The subdiagonal block $A_k$ describes the interaction between the two blocks of variables in the $k$-th matricization of $f_\Gamma$.
If $\Gamma$ is diagonal, we have $A_k = 0$ and the blocks do not interact, hence the density factorizes.
When $A_k \neq 0$, the density does not factorize exactly any more, but if the structure of $A_k$ is sufficiently simple for each $k$, we can derive bounds on the ranks necessary to approximate $f_\Gamma$.
We look at two cases: In Theorem \ref{thm:lowrankapprox}, we assume that each $A_k$ is a low-rank matrix with uniformly bounded singular values and in Theorem \ref{thm:expdecayapprox}, we assume that the singular spectrum of $A_k$ decays at an exponential rate.
The proofs of both theorems are collected in Section \ref{sec:proofs}.

For the first case, we assume that there exists a (small) $l \in \N$ such that every subdiagonal block $A_k$ has $\rank A_k \leq l$.
We assume further that $A_k$ has singular values $\sigma_i^k$, $i=1, \dots, l$, that are uniformly bounded and set
\begin{equation}
	\sigma := \max_{k, i} \sigma_i^k.
\end{equation}
\begin{theorem}\label{thm:lowrankapprox}
	Let $\Gamma$ be a symmetric positive definite precision matrix with low-rank subdiagonal blocks as described above.
	For every $\eps > 0$ there exists an approximation $\hat{f}$ to the Gaussian density $f_\Gamma$ with
	\begin{equation}
		\|f_\Gamma - \hat{f} \|_{L^2(\R^d)} \leq \eps \|f_\Gamma\|_{L^2(\R^d)}	
	\end{equation}
	whose FTT-ranks $r_k$, $k=1, \dots, d-1$, are bounded by
	\begin{align}\label{eqn:rankestimate}
		r_k &\leq \left( \left( 1 + 7 \frac{\sigma}{\lambda_{\text{min}}} \right) \log \left( \frac{\sqrt{8} d}{\eps} \right) + \log \left( e^{3/2} \frac{l}{2} \right) \right)^{l} \\
			&\leq \left( \left( 1 + 7 \frac{\sigma}{\lambda_{\text{min}}} \right) \log \left( 7 l \frac{d}{\eps} \right)\right)^{l}
	\end{align}
    where $\lambda_{\text{min}}$ is the smallest eigenvalue of $\Gamma$.
\end{theorem}
In the situation described above, the FTT ranks only grow poly-logarithmically in dimension $d$ and accuracy $1/\eps$.
By working with a tensor approximation of the density $f_\Gamma$, we can therefore avoid the curse of dimensionality.
There is however a critical exponential dependence on the number of singular values $l$ in the subdiagonal blocks.

In statistical applications, the structure of a Gaussian distribution is often more conveniently described by its covariance matrix (i.e.\ by the inverse of the precision matrix).
In this case, Theorem \ref{thm:lowrankapprox} is still applicable since the rank of the subdiagonal blocks are the same for $\Gamma^{-1}$ and $\Gamma$ \cite{Fiedler1993StructureRO}.
As an important special case, we can apply Theorem \ref{thm:lowrankapprox} whenever the covariance or the precision matrix is sparsely populated with entries near to the diagonal (e.g. a band matrix).

The strict low-rank requirement can be replaced by a strong decay rate for the singular values.
We assume that for all $k$-matricizations, the following exponential decay property for the singular values of the corresponding subdiagonal block holds
\begin{equation}\label{eqn:sigmadecayrate}
	\sigma^k_i \leq \alpha e^{-\theta i}.
\end{equation}
\begin{theorem}\label{thm:expdecayapprox}
	Let $\Gamma$ be a symmetric positive definite precision matrix with exponentially decaying singular values in each subdiagonal block as described above.
    For every $\eps > 0$ there exists an approximation $\hat{f}$ to the Gaussian density $f_\Gamma$ with
	\[
		\|f_\Gamma - \hat{f} \|_{L^2(\R^d)} \leq \eps \|f_\Gamma\|_{L^2(\R^d)}	
	\]
	whose FTT-ranks $r_k$, $k=1, \dots, d-1$, are bounded by
	\begin{align}
		r_k &\leq \exp\left(\frac{3 \alpha}{\lambda_{\min} \theta}\right) 
                    \left(3 \log \left( C \frac{d}{\eps} \right) \right)^{\frac{2}{\theta} \log \left( C \frac{d}{\eps} \right)} \label{eqn:thm2_eq1} \\
            &= \exp\left(\frac{3 \alpha}{\lambda_{\min} \theta}\right) 
                \left( C \frac{d}{\eps} \right)
                    ^{\frac{2}{\theta} \log \left( 3 \log \left( C \frac{d}{\eps}\right) \right)} \label{eqn:thm2_eq2}
	\end{align}
    with
    \[
        C := \max \left \{ \sqrt{8}, \frac{5}{\theta}, \frac{e^{\theta}}{1 + e^{\theta}}{\frac{4 \alpha}{ \lambda_{\text{min}}}} \right\}
    \]
    where $\lambda_{\min}$ is the smallest eigenvalue of $\Gamma$.
\end{theorem}
Compared to (\ref{eqn:rankestimate}) from Theorem \ref{thm:lowrankapprox}, the exponent in (\ref{eqn:thm2_eq1}) now has a logarithmic dependence on the dimension $d$ and accuracy $1/\eps$.
The reformulation in (\ref{eqn:thm2_eq2}) shows that this cannot be
neglected: If we ignore the log-log term in the exponent, the equation
states a polynomial growth rate of the ranks in the dimension $d$ and the accuracy $1/\eps$ which is substantially worse than the poly-logarithmic rate in Theorem \ref{thm:lowrankapprox}.
This is consistent with our interpretation of the complexity associated with a Gaussian density.
Given a certain target accuracy, we can ignore all interactions between the two variable blocks of a matricization described by singular values below a certain threshold since their influence on the density is negligible.
As we increase the approximation accuracy, we need to refine our resolution of the dependencies between the blocks of the matricizations, which, according to Theorem \ref{thm:lowrankapprox}, results in a logarithmic growth of the rank per singular value we look at.
However, we also need to decrease the threshold, thereby increasing the number of singular values we need to account for.
We will see in Section \ref{sec:proofs} that the number of singular values we have to look at also grows logarithmically in $1/\eps$.
Combined, this indicates a rate of roughly
\[
   \log(1/\eps)^{\log(1/\eps)} = (1/\eps)^{\log \log(1/\eps)}
\]
(for $\eps \ll 1$) which is similar to the bound in the theorem.
Similarly to the dependence on the rank in the previous theorem, this bound depends very sensitively on the decay rate parameters $\alpha, \theta$ in (\ref{eqn:sigmadecayrate}).

In both theorems, we have described the FTT ranks by looking at various subdiagonal blocks which describe the correlation between variables of the two dimensions of the matricized function.
The structure of these blocks, as well as the ranks necessary for the Tensor-Train approximation, therefore depend on the ordering of the variables.
Furthermore, our results can be easily extended to the more general hierarchical Tucker (HT) format \cite{grasedyck2010hierarchical}.
The low-rank manifold of HT tensors is, similar to the TT format, described by the ranks of the linear subspaces induced by a generalized notion of matricizations.
As with the TT format (Theorem \ref{thm:ttapproximation}), approximability in the HT format can be established by approximating the corresponding matricizations.
By analyzing the generalized subdiagonal blocks that correspond to the matricizations in the HT format, one can proof equivalent results to Theorem \ref{thm:lowrankapprox}, \ref{thm:expdecayapprox}.

\subsection{Proofs of the Theorems}
\label{sec:proofs}

We have seen in (\ref{eqn:diagdensity}) that for diagonal precision matrices, the density is of rank $1$ since we can write it as product of one dimensional functions.
For a general precision matrix, such a multiplicative decomposition would contain additional exponential factors of $2$ variables.
The general idea behind the proofs of our theorems is to approximate these functions individually with low rank.
However, a naive approximation of each function would lead to a rank bound that grows exponentially in the number of off-diagonal terms in $\Gamma$ which generally grows at least linear in the dimension.
In the following, we use the properties of the TT format, in particular Theorem \ref{thm:fttapproximation}, as well as suitable coordinate transformations to reduce the number of functions that we need to approximate.

For a general precision matrix, the multiplicative decomposition of the density reads
\begin{equation}
	f_\Gamma(x) = \prod_{i,j = 1}^{d} e^{\frac{1}{2}\Gamma_{i,j} x_i x_j}.
\end{equation}
Each of the factors is a two dimensional function of the form
\[
	(x, y) \mapsto e^{-\gamma x y} \;\; \text{ for some } \gamma.
\]
Interpolating this exponential function with a one dimensional polynomial $p(t) = \sum_{i=1}^{r} c_i t^{i-1} \in \mathbb{P}_r$ of order $r$ (and therefore of polynomial rank $r-1$) yields an approximation
\[
	e^{-\gamma x y} \approx p(x \cdot y) = \sum_{i=1}^{r} c_i x^{i-1} \cdot y^{i-1}
\]
which is a sum of $r$ separated functions and therefore has FTT ranks bounded by $r$.
Before we can apply polynomial interpolation, we need to restrict the function $f_\Gamma$ to a finite domain.
This can be achieved by multiplying with the restriction function 
\begin{equation}
	\mathds{1}_\Omega(x) = \begin{cases}
								1, & \text{ if } x \in \Omega \\
								0, & \text{ otherwise}
							\end{cases}
\end{equation}
which has FTT rank $1$ if we restrict to a rectangular domain $\Omega = I_1 \times \dots \times I_d$.
The overall error is therefore determined by two terms, the error due to the cutoff and the error due to the inexact interpolation, which need to be balanced.
We could now apply a cutoff function and approximate each factor of the Gaussian $f_\Gamma$ with bounded rank, however, the exponential growth in the number of terms would make the resulting bound useless.
 
Using Theorem \ref{thm:fttapproximation}, we only need to prove that we can approximate the individual $k$-matricizations
\begin{equation}
	f_\Gamma^k: \R^k \times \R^{d-k} \to \R.
\end{equation}
For now, we fix $k \in \{1, \dots, d-1\}$.
The $k$-matricization of any Gaussian density $f_{\Gamma_0}$ whose precision matrix has the block structure
\begin{equation}
	\Gamma_0 := \begin{bmatrix}
			\Gamma_1 & 0 \\
			0        & \Gamma_2
		\end{bmatrix},
	\;\; \Gamma_1 \in \R^{k \times k}, \; \Gamma_2 \in \R^{(d-k) \times (d-k)}
\end{equation}
has rank $1$ as there is no correlation between the variable blocks $(x_1, \dots, x_{k})$ and $(x_{k+1}, \dots, x_d)$.
This motivates splitting the precision matrix in the diagonal blocks $\Gamma_1, \Gamma_2$, and the remaining subdiagonal block $A$
\begin{equation}\label{eqn:decomp_of_prec_matrix}
	\Gamma = \begin{bmatrix}
				\Gamma_1 & A^\intercal \\
				A & \Gamma_2
			 \end{bmatrix},
         \;\; A \in \R^{(d-k) \times k}.
\end{equation}
At this stage, we would need to approximate an exponential for every non-zero entry in $A$ which are in general substantially fewer factors then every non-zero off-diagonal entry of $\Gamma$.

However, we can reduce the number factors further by rotating the coordinate system for the first $k$ variables $\{x_1, \dots, x_k\}$ and the last $d-k$ variables $\{x_{k+1}, \dots, x_d\}$ respectively.
Let $A = U \Sigma V^\intercal$ be a singular value decomposition of $A$ and 
\begin{equation}
	\Gamma = \begin{bmatrix} V \\ & U \end{bmatrix} 
         \begin{bmatrix}\tilde\Gamma_1 & \Sigma^\intercal \\ \Sigma & \tilde\Gamma_2 \end{bmatrix}
         \begin{bmatrix}V^\intercal \\ & U^\intercal \end{bmatrix},
\end{equation}
where $\tilde\Gamma_1 = V^\intercal \Gamma_1 V$, $\tilde\Gamma_2 = U^\intercal \Gamma_2 U$.
We approximate the function $f_\Gamma$ from (\ref{eqn:def_f_gamma}) on the transformed coordinates
\begin{equation}\label{eqn:domaintransform}
	\tilde x_i = \left( \begin{bmatrix} V \\ & U \end{bmatrix}^\intercal x \right)_i
		= \begin{cases}
			(V^\intercal [x_1 \, \dots \, x_k]^\intercal)_i & \text{for } i \leq k, \\
			(U^\intercal [x_{k+1} \, \dots \, x_d]^\intercal)_i & \text{otherwise.}
		\end{cases}
\end{equation}
That is, we define the transformed precision matrix
\begin{equation}
	\tilde{\Gamma} := 
	\begin{bmatrix}\tilde\Gamma_1 & \Sigma^\intercal \\ \Sigma & \tilde\Gamma_2 \end{bmatrix}
\end{equation}
and approximate the corresponding unnormalized density
\begin{equation}
	f_{\tilde{\Gamma}}(\tilde x) = e^{-\frac{1}{2} \tilde x^\intercal \tilde \Gamma \tilde x} 
	     = e^{-\frac{1}{2} x^\intercal \Gamma x} =  f_{\Gamma}(x).
\end{equation}
As before, we decompose $\tilde \Gamma$ into a matrix containing the diagonal blocks $\tilde \Gamma_0$ and a matrix containing the off-diagonal blocks (odb) $\tilde \Gamma_{\text{odb}}$
\begin{equation}
	\tilde{\Gamma} = \begin{bmatrix}
						\tilde{\Gamma}_1 & 0 \\
						0        & \tilde{\Gamma}_2
					\end{bmatrix} + \begin{bmatrix}
						0 & \Sigma^\intercal \\
						\Sigma & 0
		\end{bmatrix} =: \tilde{\Gamma}_0 +  \tilde{\Gamma}_{\text{odb}}
\end{equation}
and split the transformed density in its rank $1$ factor $ f_0 := f_{\tilde{\Gamma}_0}$ and the product of the remaining factors  $f_{\text{odb}} = f_{\tilde{\Gamma}_{\text{odb}}}$.
The latter factor yields the decomposition
\begin{align*}
	f_{\text{odb}}^k(\tilde x_1 \dots, \tilde x_{k}; \tilde x_{k+1}, \dots, \tilde x_d) 
		&= e^{-\frac{1}{2} \tilde x^\intercal \tilde{\Gamma}_{\text{odb}} \tilde x} \\
		&= \exp \left( - [\tilde x_1 \, \dots \, \tilde x_k] \Sigma [\tilde x_{k+1} \, \dots \, \tilde x_d]^\intercal \right) \\
		&= \prod_{i=1}^{\rank A} \exp \left( -\sigma_i \cdot \tilde x_i \cdot \tilde x_{k+i} \right) \\
		&=: f_1 \cdot f_2 \cdots f_{\rank A}, \numberthis\label{eqn:reprfsdb}
\end{align*}
where $\sigma_i := \Sigma_{i, i}$.
Let
\begin{equation}
	f_i(\tilde x) \approx h(\tilde x) = \sum_{i=1}^{r} h_j^{(1)}(\tilde x_i) h_j^{(2)}(\tilde x_{k+i})
\end{equation}
be a rank $r$ approximation of $f_i$ in the $k$-th matricization.
Due to the block structure of the variable transform in (\ref{eqn:domaintransform}), $h$ is also a low-rank approximation when retransformed to the original coordinate system
\begin{equation}
	f_i(x) \approx h(x) = \sum_{j=1}^{r} h_j^{(1)}\left( (V^{\intercal} [x_1  \dots x_k]^\intercal )_i \right) h_j^{(2)}\left( (U^{\intercal} [x_{k+1} \dots x_d]^\intercal)_i \right).
\end{equation}
Furthermore, the change of coordinates preserves the $L^2$ norm since the transformation matrix is orthogonal by construction.
It is therefore equivalent to show low-rank approximability for each matricization in the original and the transformed coordinate system.

To compute the low-rank approximation, we restrict to a finite domain in the transformed coordinate system and subsequently approximate each factor $f_i$ by a polynomial.
To achieve an overall relative error of $\eps$ for the $k$-th matricization, we allow an error of $\eps/2$ when restricting the domain and an error of $\eps/2$ for the subsequent low-rank approximation.
As a simplification, we restrict $f_{\tilde{\Gamma}}$ to a square domain $\Omega = [-a, a]^d$ which has to be chosen large enough such that
\begin{equation}
	\left\|f_{\tilde{\Gamma}} - f_{\tilde{\Gamma}} {\big \vert}_{\Omega} \right\|_{L^2(\R^d)}
		= \|f_{\tilde{\Gamma}} \|_{L^2(\R^d \setminus \Omega)}
		\leq \frac{\eps}{2} \left\|f_{\tilde{\Gamma}}\right\|_{L^2(\R^d)} 
		= \frac{\eps}{2} \left\|f_{\Gamma}\right\|_{L^2(\R^d)}.
\end{equation}
This is equivalent to determining a square domain that covers $1-(\eps/2)^2$ of the volume of the random variable
\[ 
	X \sim \mathcal{N}\left(0, (2 \tilde{\Gamma})^{-1}\right).
\]
This can be ensured by choosing $a$ such that every marginal density covers at least $1-(\eps/2)^2/d$ of the weight.
Given $\Gamma$, the minimal decay rate that can occur is bounded from below by the minimal eigenvalue
\[
	\lambda_{\text{min}}(2 \tilde{\Gamma}) = 2 \lambda_{\text{min}}(\Gamma).
\]
Therefore, it is sufficient to choose $a$ larger than the quantile function of the normal distribution for $\sigma^2 = \frac{1}{2 \lambda_{\text{min}}(\Gamma)}$, i.e.
\[
	a \geq \frac{1}{\sqrt{\lambda_{\text{min}}}} \text{erfc}^{-1}\left( \frac{1}{d} \left( \frac{\eps}{2}\right)^2 \right).
\]
For simplicity, we bound the inverse complementary error function as in \cite{chiani2002improved} which yields
\begin{equation}\label{eqn:defa}
	\frac{1}{\sqrt{\lambda_{\text{min}}}} \text{erfc}^{-1}\left( \frac{1}{d} \left( \frac{\eps}{2}\right)^2 \right) 
		\leq \sqrt{\frac{2}{\lambda_{\text{min}}} \log \left( \frac{\sqrt{2d}}{\eps} \right)} =: a.
\end{equation}
Now, each function $f_i$, treated as a univariate function $x \mapsto e^{-\sigma_i x}$, needs to be approximated by a polynomial on the domain
\[
	\widetilde{\Omega} = [-a^2, a^2].
\]
In general, the subblock matrices $A$ could have full rank and, assuming order $\mathcal{O}(1)$ interpolations of the factors, the rank of the approximation still increases exponentially in the dimension $d$.
This means that without structural assumptions on the precision matrix, we cannot expect our approach to show good FTT approximability.
The decomposition (\ref{eqn:reprfsdb}) of $f_{\text{odb}}$ leads to two natural conditions that result in approximable densities, namely restricting the overall rank (Theorem \ref{thm:lowrankapprox}) or ensuring a fast enough decay of the singular values (Theorem \ref{thm:expdecayapprox}).

Let $l := \rank A$ be the rank of the subdiagonal block $A$ in (\ref{eqn:decomp_of_prec_matrix}) for the $k$-th matricization. 
We approximate each function $f_i$ in (\ref{eqn:reprfsdb}) by a polynomial $p_i$.
To bound the overall error made by doing this, we single out each error pair 
\[
	f_1 \cdots f_l - p_1 \cdots p_l = \sum_{i=1}^{l} p_{1} \cdots p_{i-1} (f_i - p_i) f_{i+1} \cdots f_l
\]
to get
\begin{align}
	E &:= \|f_0 [ f_1 \cdots f_l - p_1 \cdots p_l ] \|_{L^2(\Omega)}  \label{eqn:error1} \\
	  &\leq \sum_{i=1}^{l} \|f_i - p_i \|_{L^\infty(\Omega)} \|f_0 p_1 \cdots p_{i-1} f_{i+1} \cdots f_l \|_{L^2(\Omega)}. \label{eqn:error2}
\end{align}
The $L^\infty$ norm of the polynomial approximation is independent of the dimensionality of the underlying space $\Omega$ and only depends on the approximation quality of the exponential function on the one dimensional domain $\widetilde{\Omega}$.
For simplicity of notation, we identity $f_i$ with its corresponding one dimensional exponential function
\begin{equation}\label{eqn:def_f_1d}
	f_i: \widetilde{\Omega} \to \R, x \mapsto e^{-\sigma_i x}.
\end{equation}
To derive bounds for the error, we apply classical polynomial interpolation theory.
\begin{lemma}\label{lemma:polynomialapprox}
	Let $x_1, \dots, x_r$ be the roots of the $(r-1)$-th Chebyshev polynomial transformed to the interval $\widetilde{\Omega}$. 
	The error of the order $r$ interpolation polynomial $p_i$ of $f_i$ with nodes $x_1, \dots, x_r$ is bounded by
	\begin{align}
		\max_{\widetilde{\Omega}} |f_i - p_{i}|
		&\leq \frac{1}{r! \cdot 2^{r-1}} \left( \sigma_i a^2 \right)^{r} e^{ \sigma_i a^2} \\
		&\leq \left( \frac{\sigma_i e a^2}{2 r} \right)^r e^{ \sigma_i a^2}.
	\end{align}
\end{lemma}
\begin{proof}
	The error of a order $r$ polynomial interpolation on the Chebyshev nodes $x_1, \dots, x_r$ is bounded by
	\begin{align*}\label{eqn:polapproxerr}
		|f_i(x) - p_{i}(x)| &\leq \frac{1}{r! \cdot 2^{r-1}} a^{2r}
			\max_{x \in \widetilde{\Omega}} \left| \frac{ \partial^r}{\partial x^r} e^{-\sigma_i x} \right| \\
			&\leq \frac{1}{r! \cdot 2^{r-1}} \left(\sigma_i a^2 \right)^{r} e^{\sigma_i a^2}
	\end{align*}
	(see e.g. \cite[Corollary 8.11]{burden2010numerical}).
	For the second part, we use the Sterling approximation
	\[
		\sqrt{2\pi r} \left(\frac{r}{e}\right)^r \leq r!
	\]
	to estimate the factorial term
	\[
		\frac{1}{r! \cdot 2^{r-1}} \leq \frac{1}{\sqrt{2 \pi r} \left(\frac{r}{e}\right)^r 2^{r-1}} \leq \frac{2}{\sqrt{2 \pi}} r^{-r} \left(\frac{e}{2}\right)^r \leq r^{-r} \left(\frac{e}{2}\right)^r.
	\]
\end{proof}
To get from (\ref{eqn:error2}) to an estimate of the relative error, we need to relate the $L^2$ norm of the partial approximation
\begin{equation}\label{eqn:partialapprox}
	\|f_0 \cdot p_1 \cdots p_{i-1} \cdot f_{i+1} \cdots f_l\|_{L^2(Q)}
\end{equation}
to the $L^2$ norm of the density.
As a preliminary consideration, we quantify the influence of one factor $f_i$ on the norm of the product.
\begin{lemma}\label{lemma:single_term_stability_svd}
	Let
	\[
		f^* \in L^2(\Omega)
	\]
	and $i \in \{1, \dots, l\}$. Then, it holds for any $f_i$ as defined in (\ref{eqn:reprfsdb}) that
	\begin{equation}
		\|f^* \|_{L^2(\Omega)} \leq \|f^* f_i \|_{L^2(\Omega)} \exp \left( \sigma_i a^2 \right).
	\end{equation}
\end{lemma}
\begin{proof}
	By the definition of $f_i$, we get
	\[	
		f_i(x) = \exp \left( - \sigma_i x_i x_{i+k} \right)
		\geq \exp \left(- \sigma_i a^2 \right) > 0.
	\]
	Using this, we have
	\begin{align}\label{eqn:single_term_stability_bound}
	\|f^*\|_{L^2(\Omega)} &= \| f^* \frac{f_i}{f_i} \|_{L^2(\Omega)} \leq \| f^* f_i\|_{L^2(\Omega)}
		\exp \left( \sigma_i a^2 \right).
	\end{align}
\end{proof}
While this is a crude estimate, the additional exponential term  in (\ref{eqn:single_term_stability_bound}) already comes up as part of the corresponding polynomial approximation in Lemma \ref{lemma:polynomialapprox}.

\begin{lemma}\label{lemma:stability}
	Let
	\begin{equation}\label{eqn:def_eps_stability}
		\eps^{(r_i)}_j := \frac{1}{r_j! \cdot 2^{r_j-1}} \left( \sigma_j a^2 \right)^{r_j} e^{2 \sigma_j a^2}, \;\; j = 1, \dots, l, 
	\end{equation}
	be the product of the $r_j$-th order polynomial interpolation error bound for $f_j$ on $\widetilde{\Omega}$ from Lemma \ref{lemma:polynomialapprox} and the correction factor from Lemma \ref{lemma:single_term_stability_svd}.
	The partial approximation term (\ref{eqn:partialapprox}) can be estimated by
	\begin{equation}
		\|f_0 \cdot p_1 \cdots p_{i-1} \cdot f_{i+1} \cdots f_l \|_{L^2(\Omega)} 
		\leq \prod_{j=1}^{i-1} \left( 1 + \eps^{(r_j)}_{j} \right) e^{\sigma_i a^2} 
			\|f_0 \cdot f_1 \cdots f_l \|_{L^2(\Omega)}.
	\end{equation}
\end{lemma}
\begin{proof}
	Let $f^* \in L^2(\Omega)$ and $\varrho \in \{1, \dots, i-1\}$.
	Then
	\begin{align*}
		\|f^* p_\varrho \|_{L^2(\Omega)} & \leq \|f^* f_\varrho 
			\|_{L^2(\Omega)} + \|f_\varrho - p_\varrho\|_{L^\infty(\widetilde{\Omega})} \|f^*\|_{L^2(\Omega)} \\
		&\leq \left( 1 + e^{\sigma_\varrho a^2} 
			\|f_\varrho - p_\varrho\|_{L^\infty(\widetilde{\Omega})} \right)
			\|f^* f_\varrho \|_{L^2(\Omega)} \\
		& \leq \left( 1 + \eps^{(r_\varrho)}_\varrho \right) \|f^* f_\varrho\|_{L^2(\Omega)}.
	\end{align*}
	Note that the multiplicative factor does not depend on $f^*$. 
	Therefore, we can simply extend this inductively to more than one factor
	\begin{align*}
		\|f^* p_{1} \cdots p_{k} \|_{L^2(\Omega)}
		&\leq \|f^* p_{1} \cdots p_{k-1} f_{k} \|_{L^2(\Omega)} 
			+ \|f_{k} - p_{k}\|_{L^\infty(\Omega)} \|f^* p_{1} \cdots p_{k-1} \|_{L^2(\Omega)} \\
		&\leq \prod_{j=1}^{k-1} \left( 1 + \eps_{j}^{(r_{j})} \right)
			\cdot \left( \|f^* f_{1} \cdots f_{k-1} f_{k} \|_{L^2(\Omega)} 
			+ \|f_{k} - p_{k}\|_{L^\infty(\Omega)} \|f^* f_{1} \cdots f_{k-1} \|_{L^2(\Omega)} \right) \\
		&\leq \prod_{j=1}^{k} \left( 1 + \eps_{j}^{(r_{j})} \right) \|f^* f_{1} \cdots f_{k} \|_{L^2(\Omega)}.
	\end{align*}
	Setting $f^* = f_0 f_{i+1} \cdots f_{l}$, $k = i-1$, and using Lemma \ref{lemma:single_term_stability_svd} to add the missing factor $f_i$ completes the proof.
\end{proof}
The following lemma will be useful later to estimate the required order of the interpolation polynomial.
\begin{lemma}\label{lemma:powertoexptrick}
	Let $c > 0$. Then it holds that
	\begin{equation}
		\left( \frac{c}{c + r} \right)^{c + r} \leq e^{-r}.
	\end{equation}
\end{lemma}
\begin{proof}
	We look at the difference of the logarithms of both sides
	\[
		f(r) := \log \left( \left( \frac{c}{c + r} \right)^{c + r}\right) - \log \left(e^{-r} \right)
			= (c+r) (\log(c) - \log(c+r)) + r.
	\]
	Note that $f(0) = 0$. The derivative of $f$ satisfies
	\[
		\frac{\text{d}}{\text{d} r} f(r) = \log(c) - \log(c+r) \leq 0 \;\; \text{ for } \;\; r \geq 0.
	\]
	Therefore, $f(r) \leq 0$ for $r \geq 0$ and the result follows from the monotonicity of the logarithm.
\end{proof}

\subsubsection{Proof of Theorem \ref{thm:lowrankapprox}}

At this stage, we need to start making use of the structure of the precision matrix.
We start with the low-rank case as described in Theorem \ref{thm:lowrankapprox}.
For this, we assume that for every $k$-matricization the subdiagonal blocks $A_k$ have $\rank A_k \leq l$ where the singular values $\sigma_i^k$, $i=1, \dots, l$, are uniformly bounded and set
\begin{equation}
	\sigma := \max_{k, i} \sigma_i^k.
\end{equation}
\begin{proof}[of Theorem \ref{thm:lowrankapprox}]
	Again, we fix a matricization $k \in \{1, \dots, d-1\}$ and look at the function $f_{\tilde \Gamma}(\tilde x)$ on the transformed domain as in (\ref{eqn:domaintransform}).
	As a first step, we bound the rank necessary to construct a low-rank approximation with a relative error bounded by $\epsM < 1$ for this matricization.
	We fix a finite subdomain $\Omega = [-a, a]^d$ with $a$ as in (\ref{eqn:defa}) such that the relative $L^2$ error due to the cutoff is bounded by $\epsM/2$.
	Additionally, we need to choose a polynomial order $r_i$ such that the contribution due to the approximation on $\Omega$ is bounded by $\epsM/2$.
	We choose a uniform interpolation order $r_i = r$ for all $f_i$, $i=1, \dots, l$.
	Let
	\[
		\eps^{(r)} := \frac{1}{r! \cdot 2^{r-1}} \left( \sigma a^2 \right)^{r} e^{2 \sigma a^2}.
	\]
	Using Lemma \ref{lemma:polynomialapprox} and Lemma \ref{lemma:stability}, we can bound the error $E$ when using an order $r$ polynomial interpolation by
	\begin{align*}
		E &\leq \sum_{i=1}^{l} \|f_i - p_i\|_{L^\infty(\Omega)}
			\|f_0 \cdot f_1 \cdots f_{i-1} \cdot p_{i+1} \cdots p_l \|_{L^2(\Omega)} \\
		&\leq \left(1 + \eps^{(r)}\right)^{l-1} l \eps^{(r)} \|f_0 \cdot f_1 \cdots f_l \|_{L^2(\Omega)}.
	\end{align*}
	Therefore, we need to choose $r$ large enough such that
	\begin{equation}\label{eqn:relerror_lowrank}
		\left(1 + \eps^{(r)} \right)^{l-1} l \eps^{(r)} \leq \frac{\epsM}{2}.
	\end{equation}
	For any $r$ that satisfies (\ref{eqn:relerror_lowrank}), we have $\eps^{(r)} \leq \frac{\epsM}{2 l}$.
	We can utilize this to simplify the left-hand side in (\ref{eqn:relerror_lowrank}) by using
	\[
		(1 + \eps^{(r)})^{l-1} \leq \exp \left( l \log \left( 1 + \frac{\epsM}{2l} \right) \right)
			\leq e^{\epsM/2}
            \leq e^{1/2}.
	\]
	Using Lemma \ref{lemma:polynomialapprox}, it is sufficient to choose $r$ large enough such that we can assure
	\[
		\left( \frac{\sigma e a^2}{2 r} \right)^r e^{ 2 \sigma a^2}
			\leq \frac{\epsM}{ 2 l e^{1/2}}
	\]
	which is equivalent to 
	\begin{equation}\label{eqn:twosteprankbound}
		\left( \frac{\sigma e a^2}{2 r} \right)^r \leq \frac{\epsM}{2 l e^{1/2}} e^{-2 \sigma a^2}.
	\end{equation}
    We compute the necessary interpolation order in two steps:
	First, we set $r_1 := \frac{e}{2} \sigma a^2$ which pushes the left-hand side in (\ref{eqn:twosteprankbound}) to $1$.
	Then, we can use Lemma \ref{lemma:powertoexptrick} to bound the remaining part
	\[
		\left( \frac{r_1}{r_1 + r_2} \right)^{r_1 + r_2} \leq e^{-r_2}
	\]
	and as a result compute the second part $r_2$ as
	\begin{align}
		\log \left( \frac{2 l e^{1/2}}{\epsM} e^{2 \sigma a^2} \right) 
		&\leq 2 \sigma a^2 + \log \left( l e^{1/2}  \frac{2}{\epsM} \right) \\
		& =: r_2.
	\end{align}
    Then, choosing $r = r_1 + r_2$ is sufficient to guarantee (\ref{eqn:relerror_lowrank}).
	Since we cannot interpolate with non-integer order, we need to choose the next larger integer value.
    To account for this, we simply add $1$ to our bound on $r$.
	Overall, this gives us
	\begin{align}
		r &\leq \left(\frac{e}{2} + 2\right) \sigma a^2 + \log \left( l e^{1/2}  \frac{2}{\epsM} \right)+ 1 \\
		&= \left(\frac{e}{2} + 2\right) \sigma a^2 + \log \left( l e^{3/2}  \frac{2}{\epsM} \right).
	\end{align}
	Choosing $a$ as in (\ref{eqn:defa}) such that additional error that comes from restricting to the finite domain $\Omega$ is bounded by $\epsM/2$ yields
	\begin{align}
		r &\leq  \left(\frac{e}{2} + 2\right) \frac{2 \sigma}{\lambda_{\text{min}}} \log \left( \frac{2\sqrt{2d}}{\epsM} \right) + \log \left( l e^{3/2}  \frac{2}{\epsM} \right) \\
			&\leq \left( 1 + 7 \frac{\sigma}{\lambda_{\text{min}}} \right) \log \left( \frac{\sqrt{8d}}{\epsM} \right) + \log \left( e^{3/2} \frac{l}{2} \right) \label{eqn:boundonr1} \\
			&\leq \left( 1 + 7 \frac{\sigma}{\lambda_{\text{min}}} \right) \log \left( 7 l \frac{\sqrt{d}}{\epsM} \right) \label{eqn:boundonr2}
	\end{align}
	where we used in the second step that $d \geq 2$.
	We have shown that if we restrict $f_{\tilde \Gamma}$ to the domain $\Omega$ and choose the order of our interpolation larger than (\ref{eqn:boundonr1}) or (\ref{eqn:boundonr2}), the overall error of an rank $R = r^l$ approximation on $\R^d$ is bounded by
	\begin{align*}
		\|f_{\tilde{\Gamma}} - \mathds{1}_{\Omega} f_0 p_1 \cdots p_l \|_{L^2(\R^d)}
			&\leq \| f_{\tilde{\Gamma}} - f_{\tilde{\Gamma}} \big \vert_{\Omega} \|_{L^2(\R^d)} 
				+ \|f_{\tilde{\Gamma}} - f_0 p_1 \cdots p_l \|_{L^2(\Omega)} \\
			&\leq \epsM \|f\|_{L^2(\R^d)}.
	\end{align*}
	Since we have seen that proving a low-rank bound for $f_\Gamma(\tilde x)$ on the transformed coordinate system is equivalent to proving a low-rank bound for $f_\Gamma(x)$, this holds for the rank necessary to approximate the $k$-th matricization to a specified accuracy.
    This rank bound does not depend on $k \in \{1, \dots, d-1\}$ and we can apply Theorem \ref{thm:fttapproximation} to bound the error of the FTT approximation $\hat f$ of the density by
	\begin{equation}
		\|f_{\Gamma} - \hat{f}\|_{L^2(\R^d)} \leq \sqrt{d-1} \epsM \|f_\Gamma\|_{L^2(\R^d)}.
	\end{equation}
	Setting $\epsM := \frac{\eps}{\sqrt{d}}$ completes the proof.
\end{proof}

\subsubsection{Proof of Theorem \ref{thm:expdecayapprox}}

The key of proving low-rank approximability is having few, well behaved factors $f_i$ that we can interpolate using polynomials.
For Theorem \ref{thm:lowrankapprox}, we enforced this by assuming a low rank which might be limiting in practice.
As an alternative, we look at the case of a rapidly decaying singular spectrum.
In Theorem \ref{thm:expdecayapprox}, we assume that for all $k$-matricizations, it holds
\begin{equation}
	\sigma^k_i \leq \alpha e^{-\theta i}.
\end{equation}
The idea of the proof is to neglect all singular values that are small enough to not perturb the density too much (effectively approximating the corresponding exponential function by $1$).
First, we need to determine the number of singular values that we have to look at.
\begin{lemma}\label{lemma:entriestoapprox}
    Let $k \in \{1, \dots, d-1\}$ and $0 < \eps < 1$.
    Under the assumptions made above, the error made by neglecting all terms larger than
    \begin{equation}
        l \geq \frac{1}{\theta} \log \left( \frac{e^{\theta}}{1 + e^{\theta}}{\frac{3 \alpha a^2}{2\eps}} \right)
    \end{equation}
    in the decomposition (\ref{eqn:reprfsdb}) of the $k$-th matricization of $f_\Gamma$ is bounded by
    \begin{equation}
        \|f_{\tilde \Gamma} - f_0 f_1 \dots f_{\lfloor l \rfloor}\|_{L^2(\Omega)} \leq \eps \|f_{\tilde \Gamma}\|.
    \end{equation}
\end{lemma}
\begin{proof}
    Using our decomposition of $f_\Gamma$, we have
    \begin{align*}
        \|f_{\tilde \Gamma} - f_0 f_1 \dots f_{\lfloor l \rfloor}\|_{L^2(\Omega)} 
        &= \|f_{\tilde \Gamma} (1 - (f_{\lfloor l+1 \rfloor} \cdots f_{\rank A})^{-1})\|_{L^2(\Omega)}\\
        &\leq \|1 - (f_{\lfloor l+1 \rfloor} \cdots f_{\rank A})^{-1}\|_{L^\infty(\Omega)} \| f_{\tilde \Gamma}\|_{L^2(\Omega)}.
    \end{align*}
    We need to choose $l$ large enough such that
    \[
        \big \|1 - (f_{\lfloor l+1 \rfloor} \cdots f_{\rank A})^{-1} \big\|_{L^\infty(\Omega)}
            = \big\|1 - e^{\sum_{j=\lfloor l+1 \rfloor}^{\rank A} \alpha e^{- \theta j} \tilde x_j  \tilde x_{k+j} } \big\|_{L^\infty(\Omega)}
            = e^{\sum_{j=\lfloor l+1 \rfloor}^{\rank A} \alpha e^{- \theta j} a^2 } - 1
            \leq \eps.
    \]
    Since
    \[
        e^x - 1 \leq \frac{1}{\log(2)} x \leq \frac{3}{2} x \;\; \text{for} \;\; 0 < x < \log(2)
    \]
    and $\eps < 1$ it is sufficient to choose $l$ large enough such that
    \[
        \frac{3}{2}\sum_{j=\lfloor l+1 \rfloor}^{\rank A} \alpha  e^{- \theta j} a^2 
            \leq \frac{3}{2} \alpha a^2 \sum_{j= \lfloor l+1 \rfloor}^{\infty} e^{-\theta j}
            \leq \frac{e^{\theta}}{1 + e^{\theta}}\frac{3 \alpha a^2}{2} e^{-\theta l}
            \leq \eps.
    \]
    Solving for $l$ completes the proof.
\end{proof}
In particular, the number of factors that we need to deal with explicitly grows logarithmically in the target accuracy in each matricization.
We proceed similarly as for the low-rank case to approximate these terms.

\begin{proof}[of Theorem \ref{thm:expdecayapprox}]
    We fix a matricization $k \in \{1, \dots, d-1\}$ and construct a low-rank approximation of $f_{\tilde \Gamma}$ with relative error $\epsM$.
    We choose the domain $\Omega = [-a, a]^d$ with cutoff error $\epsM/2$ as in (\ref{eqn:defa}).
    Let
    \begin{equation} \label{eqn:defl}
        l := \frac{1}{\theta} \log \left( \frac{e^{\theta}}{1 + e^{\theta}}{\frac{6 \alpha a^2}{\epsM}} \right).
    \end{equation}
    We ignore all singular values $\sigma_i^k$ for $i > l$.
    This introduces an error which is controlled using Lemma \ref{lemma:entriestoapprox} 
    \[
        \|f_{\tilde \Gamma} - f_0 f_1 \cdots f_{\lfloor l \rfloor} \|_{L^2(\Omega)} \leq \frac{\epsM}{4} \|f_{\tilde \Gamma} \|_{L^2(\Omega)}.
    \]
    It remains to choose ranks for the polynomial interpolation of the $\lfloor l \rfloor$ remaining terms.
    Let
    \begin{equation}\label{eqn:errorbound_expdecay}
    	\eps^{(r_i)}_i := \frac{1}{r_i! \cdot 2^{r_i-1}} \left( \alpha e^{-\theta i} a^2 \right)^{r_i} e^{ 2 \alpha e^{-\theta i} a^2}
    \end{equation}
    be the product of the interpolation error and one-term norm correction.
	Similar to before, we can bound the overall error $E$ using Lemma \ref{lemma:stability} and thus need to choose the ranks such that
	\begin{align*}
		E &\leq \sum_{i=1}^{\lfloor l \rfloor} \prod_{j=1}^{i-1} \left(1 + \eps_j^{(r_j)}\right) 
			e^{\sigma_i a^2} \|f_i - p_i\|_{L^\infty(\Omega)} \|f_{\tilde{\Gamma}}\|_{L^2(\Omega)} \\
		&\leq \prod_{i=1}^{\lfloor l \rfloor} \left( 1 + \eps^{(r_i)}_{i} \right) 
			\sum_{i=1}^{\lfloor l \rfloor} \eps^{(r_i)}_{i} \|f_{\tilde{\Gamma}}\|_{L^2(\Omega)}   \\
		&\leq \exp \left( \sum_{i=1}^{\lfloor l \rfloor} \eps_{i}^{(r_i)} \right) 
		 \sum_{i=1}^{\lfloor l \rfloor} \eps^{(r_i)}_{i} \|f_{\tilde{\Gamma}}\|_{L^2(\Omega)} 
			\leq \frac{\epsM}{4} \|f_{\tilde{\Gamma}}\|_{L^2(\Omega)}. \numberthis\label{eqn:errorexpdecay}
	\end{align*}
	We can immediately see that for any sequence of $\left( \eps_i^{(r_i)} \right)_i$ that satisfies (\ref{eqn:errorexpdecay}), it holds
	\[
		\sum_{i=1}^{\lfloor l \rfloor} \eps^{(r_i)}_{i} \leq \frac{1}{4} \;\; \text{ and } \;\; 4 e^{\frac{1}{4}} \leq 6.
	\]
	Therefore, it is sufficient to choose $(r_i)_i$ large enough such that
	\[
		\sum_{i=1}^{\lfloor l \rfloor} \eps_i^{(r_i)} \leq \frac{\epsM}{6}.
	\]
    In the following, we assume that $l > 1$, otherwise, the bound in Theorem \ref{thm:expdecayapprox} holds trivially.
	For simplicity, we assume equidistributed errors
	\begin{equation}
		\eps_i^{(r_i)} \leq \frac{\epsM}{6 l } \;\; \text{for all} \;\; 1 \leq i \leq \lfloor l \rfloor.
	\end{equation}
	We use the bound on $\eps_i^{(r)}$ from Lemma \ref{lemma:polynomialapprox}.
    For each $i$, we have to choose the order $r_i$ big enough such that
	\[
		\left( \frac{\alpha e^{-\theta i} e a^2}{2 r} \right)^r e^{ 2 \alpha e^{-i \theta} a^2} \leq \frac{\epsM}{6l}.
	\]
	As before, we compute the order in two parts $r_i =  r_i^{(1)} + r_i^{(2)}$.
	First, we set
	\[
		r_i^{(1)} := \frac{e}{2} \alpha e^{-\theta i} a^2.
	\]
	Again, we use Lemma \ref{lemma:powertoexptrick} to compute $r_i^{(2)}$ by ensuring that 
	\[
		e^{-r_{i}^{(2)}} \leq \frac{\epsM}{6l} e^{ -2 \alpha e^{-\theta i} a^2}
	\]
	which gives us
	\[
		r_i^{(2)} := 2 \alpha e^{-\theta i} a^2 + \log \left( \frac{6 l}{\epsM} \right).
	\]
	Together, we get an upper bound for the necessary integer order on level $i$ by
	\begin{align*}
		r_i &\leq \left(2 + \frac{e}{2}\right) \alpha e^{-\theta i} a^2 + \log \left( \frac{6 l}{\epsM}\right) + 1 \\
		&\leq \frac{7\alpha}{\lambda_{\text{min}}} e^{-\theta i} \log \left( \frac{\sqrt{8d}}{\epsM} \right)
			+ \underbrace{\log \left( \frac{6}{\epsM \theta} \log \left( \frac{e^{\theta}}{1 + e^{\theta}}{\frac{6 \alpha a^2}{\epsM}} \right) \right) + 1}_{=: A}.
    \end{align*}
    Using the identity $\log(x) \leq x - 1$, we simplify the second term $A$ further 
    \begin{align}
        A &\leq \log \left( \frac{6}{\epsM \theta} \right)
                + \log \left( \frac{e^{\theta}}{1 + e^{\theta}}{\frac{6 \alpha a^2}{\epsM}} \right) 
                \label{eqn:boundofA} \\
            &\leq \log \left( \frac{6}{\epsM \theta} \right)
                + \log \left( \frac{e^{\theta}}{1 + e^{\theta}}{\frac{12 \alpha}{\epsM \lambda_{\text{min}}}} \right)
                + \log \log \left( \frac{\sqrt{8d}}{\epsM} \right) \\
            &\leq \log \left( \frac{6}{\epsM \theta} \right)
                + \log \left( \frac{e^{\theta}}{1 + e^{\theta}}{\frac{12 \alpha}{e \epsM \lambda_{\text{min}}}} \right)
                + \log \left( \frac{\sqrt{8d}}{\epsM} \right).
    \end{align}
    Thus
    \begin{align*}
        r_i &\leq \frac{7\alpha}{\lambda_{\text{min}}} e^{-\theta i} \log \left( \frac{\sqrt{8d}}{\epsM} \right)
            + \log \left( \frac{6}{\epsM \theta} \right)
            + \log \left( \frac{e^{\theta}}{1 + e^{\theta}}{\frac{12 \alpha}{e \epsM \lambda_{\text{min}}}} \right)
            + \log \left( \frac{\sqrt{8d}}{\epsM} \right) \\
            &\leq \frac{7\alpha}{\lambda_{\text{min}}} e^{-\theta i} \log \left(C \frac{\sqrt{d}}{\epsM} \right)
                + 3 \log \left( C \frac{\sqrt{d}}{\epsM} \right) \\
            &\leq  \left( 1 + \frac{3 \alpha}{\lambda_{\text{min}}} e^{-\theta i} \right) 
                        3 \log \left( C \frac{\sqrt{d}}{\epsM} \right)
    \qquad \quad \text{with} \quad
        C := \max \left \{ \sqrt{8}, \frac{5}{\theta}, \frac{e^{\theta}}{1 + e^{\theta}}{\frac{4 \alpha}{ \lambda_{\text{min}}}} \right\}
    \end{align*}
    independent of dimension $d$ and accuracy $\eps$ (where we used that $d \geq 2$).
    Comparing the definition of $l$ in (\ref{eqn:defl}) with (\ref{eqn:boundofA}), we can see that this computation additionally yields
    \[
        l \leq \frac{2}{\theta} \log \left( C \frac{\sqrt{d}}{\epsM} \right).
    \]
	The overall rank is therefore bounded by
	\begin{align*}
		R \ &\leq \prod_{i=1}^{\lfloor l \rfloor } r_i
		\leq \prod_{i=1}^{\lfloor l \rfloor} \left( \left( 1 + \frac{3 \alpha}{\lambda_{\text{min}}} e^{-\theta i} \right) 
            3 \log \left( C \frac{\sqrt{d}}{\epsM} \right) \right) \\
		&\leq \left( \prod_{i=1}^{\infty}\left(1 + \frac{3 \alpha}{\lambda_{\text{min}}} e^{-\theta i} \right) \right)
			\left(3  \log \left( C \frac{\sqrt{d}}{\epsM} \right) \right)^{l}\\
		&\leq \exp\left(\frac{3 \alpha}{\lambda_{\min} \theta}\right) 
            \left(3 \log \left( C \frac{\sqrt{d}}{\epsM} \right) \right)^{\frac{2}{\theta} \log \left( C \frac{\sqrt{d}}{\epsM} \right)} \\
        &
= \exp\left(\frac{3 \alpha}{\lambda_{\min} \theta}\right) 
            \left( C \frac{\sqrt{d}}{\epsM} \right)
                ^{\frac{2}{\theta} \log \left( 3 \log \left( C \frac{\sqrt{d}}{\epsM}\right) \right)}
	\end{align*}
    where we have used
    \[
        \prod_{i=1}^{\infty}\left(1 + \frac{3 \alpha}{\lambda_{\text{min}}} e^{-\theta i} \right)
        \leq \exp \left( \sum_{i=1}^{\infty}\frac{3 \alpha}{\lambda_{\text{min}}} e^{-\theta i}\right)
        = \exp \left( \frac{3 \alpha}{\lambda_{\min} (e^{\theta} - 1)} \right)
        \leq \exp\left(\frac{3 \alpha}{\lambda_{\min} \theta}\right) 
    \]
    in the last inequality.
	The rest of the proof follows the same steps as the proof of Theorem \ref{thm:lowrankapprox}.
\end{proof}

\section{Numerical Tests of the Rank Bounds}
\label{sec:numerics}

In this section, we present numerical tests of the FTT rank bounds from Theorem \ref{thm:lowrankapprox} and Theorem \ref{thm:expdecayapprox}.
For a choice of precision matrix, we compute a high accuracy representation of the multivariate Gaussian density as described in Section \ref{sec:computingftt} and, subsequently, check the FTT ranks that occur after compressing this representation to various accuracies.
Our aim here is to test if we can reproduce the qualitative statements made by our theorems.
We focus on the dependence of the FTT ranks on the dimension, the approximation accuracy and rank/the decay rate of the singular values in the subdiagonal blocks.

To create a fair test, we avoid any structure in the precision matrices other than that required in the respective theorems, randomizing the precision matrices while fixing the singular spectrum of the subdiagonal blocks to a prescribed sequence.
These matrices are generated as follows: starting with $M \in \R^{d \times d}$, with elements randomly chosen between $-1$ and $1$, we symmetrize by taking $\tilde M := M M^\intercal$ and then modify the subdiagonal blocks in $\tilde M$ by replacing the singular values in the block with the predefined sequence.
We alternate between the different subdiagonal blocks until the singular spectrum has converged.
To fix a uniform minimal eigenvalue, we add a suitably scaled identity matrix.

Since we can only approximate the function $f_\Gamma$ on the finite domain $Q = [-a, a]^d$, we need to choose an appropriate value for $a$.
Ideally, $Q$ should be big enough such that the $L^2$-error introduced by the cutoff is below the approximation accuracy that we are interested in.
For accuracies up to $\eps = 10^{-13}$, we would need to choose a large value for $a$ and thus a very high numbers of interpolation nodes $n$.
On large parts of this domain, the density is effectively zero which makes the tensor completion challenging.
However, in practice the ranks do not appear to depend strongly on the domain size $a$ even at high accuracy, see Figure \ref{fig:test-dependence} for an example.
We therefore choose a fixed value of $a=7$ for all following tests which results in a typical cutoff error between \num{3e-5} and \num{1e-6}, depending on the specific example.
\begin{figure}[ht]
    \center
    \includegraphics[width=0.95\linewidth]{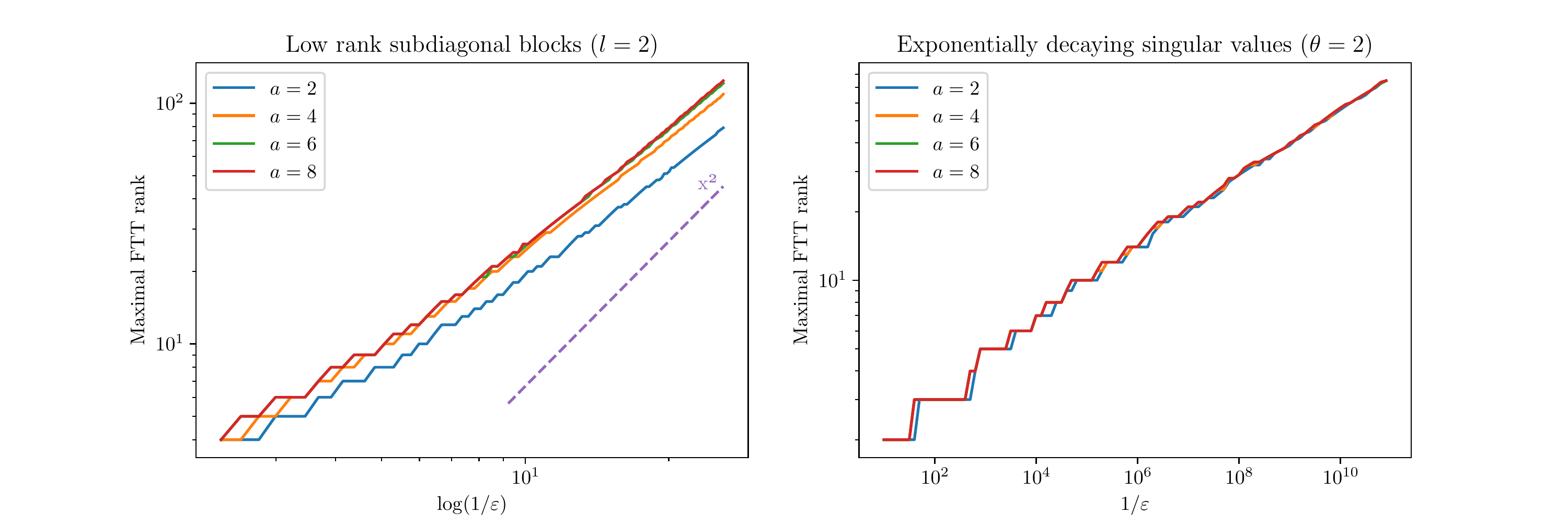}
    \caption{FTT ranks of the approximations of two exemplary Gaussian densities for varying domain sizes $Q = [-a, a]^{d}$.
    Both precision matrices are constructed randomly, as described in Section \ref{sec:numerics-low-rank} (left) and in Section \ref{sec:numerics-exp-decay} (right).
    On the left, we have set the subdiagonal rank to $l=2$ and and $d=15$. In addition, we have plotted a slope line of the predicted quadratic growth rate in $\log(1/\eps)$ (dotted line). On the right, we have set the decay rate to $\theta = 2$ and $d=30$.}
    \label{fig:test-dependence}
\end{figure}

\subsection{Dimension}

First, we fix a set of parameters for both theorems and vary the number of dimensions $d$ of the Gaussian density.
For the low-rank case, we set $\sigma \equiv 1$ and the rank of the subdiagonal blocks to~$2$.
We set the decay parameters of the exponentially decaying singular values to $\theta = 1, \alpha = 1$.
In both cases, we set the minimal singular value of the precision matrix to $\lambda_{\text{min}} = 0.5$.
We vary the number of dimensions between $d=5$ and $d=40$ and, for each choice, compute $10$ realization with a randomized precision matrix as described above.
Each Gaussian density is approximated with a relative accuracy of $\eps = 10^{-4}$.
The resulting FTT ranks can be seen in Figure \ref{fig:test-dim}.
\begin{figure}[t]
    \center
    \includegraphics[width=0.95\linewidth]{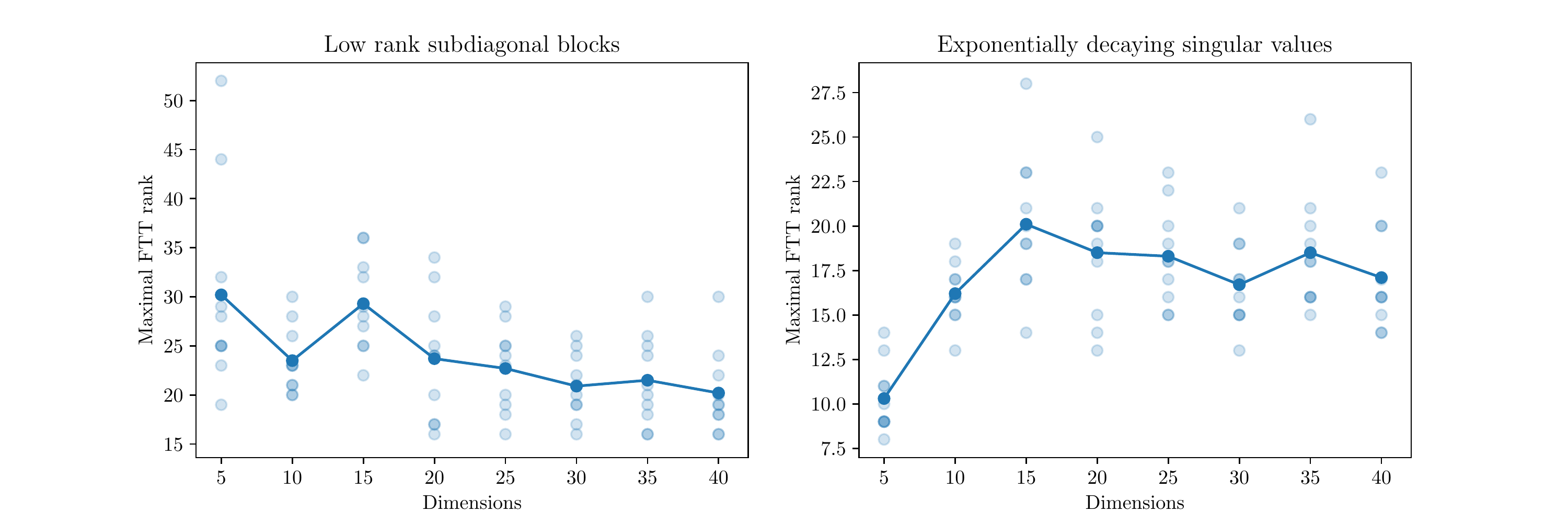}
    \caption{FTT ranks of approximations of a Gaussian density with randomized precision matrix at fixed relative accuracy $\eps = 10^{-4}$ for varying number of dimensions $d$.
    The ranks of the individual completions (translucent) as well as their averages (opaque) are shown.
    On the left, the precision matrices have subdiagonal blocks of rank $2$ with singular values $\sigma \equiv 1$.
    On the right, the singular spectrum decays with rate $e^{-j}$.
    In both cases, $\lambda_{\text{min}} = 0.5$.}
    \label{fig:test-dim}
\end{figure}

In both Theorem \ref{thm:lowrankapprox} and Theorem \ref{thm:expdecayapprox}, we predict an increase of the rank as the dimension grows.
We cannot reproduce this behavior in our tests.
Instead, it appears that the rank stays at a similar level or reduces slightly as $d$ grows.
The dimension $d$ appears in two places of the derivation of our results in Section \ref{sec:proofs}: when choosing the cutoff domain $\Omega$ and when applying Theorem \ref{thm:ttapproximation}.
The choice of $\Omega$ as a square domain is clearly suboptimal, however, the application of Theorem \ref{thm:ttapproximation} is fundamental to our approach.
This shows that any estimate that is based on approximating matricizations is not suitable to produce a dimension-independent estimate.

For the precision matrices with exponentially decaying singular values in the subdiagonal blocks, a sharp increase of the rank between $d=5$ and $d=15$ occurs which does not follow the general trend outlined above.
This could potentially be explained by the small row/column length of the subdiagonal blocks in low dimensions.
For small $d$ this length could be substantially smaller than $l$ in Lemma \ref{lemma:entriestoapprox}, leading to an initial increase of the number of relevant singular values in the subdiagonal blocks as the dimension increases.

\subsection{Low-Rank Subdiagonal Blocks}
\label{sec:numerics-low-rank}

Now, we fix the dimension and look at the FTT ranks when varying the structure of the subdiagonal blocks and the target accuracy.
We start with the case of a fixed rank $l$ in each subdiagonal block of the precision matrix.
We fix all singular values to be uniformly $\sigma \equiv 1$, set the dimension to $d=15$, and the minimal eigenvalue of the precision matrix to $\lambda_{\text{min}} = 0.5$.
We compute the approximation of $10$ different realizations of precision matrices for the subdiagonal ranks $l \in \{1, 2, 3, 4\}$. 
Due to memory constraints, the target accuracy of the TT-cross algorithm is varied depending on the subdiagonal rank.
The number of interpolation points $n$ varies between 140 and 270, depending on the target accuracy.
In all cases, we make sure that the sampled relative error on $Q$ is below the target accuracy.
From Theorem~\ref{thm:lowrankapprox}, we expect a polynomial growth rate in the logarithm of the inverse accuracy.
To explore this, we look at a log-log plot of the FTT ranks $r$ and the accuracy $\log(1/\eps)$ (meaning we have a log-log scaled x-axis).

In Figure \ref{fig:test-low-rank}, we present the results of this test.
To emphasize the growth behavior for the different parameters, we plot the average of the maximal FTT ranks of the approximations (with individual realizations as translucent lines).
We can see that at least for our random set of examples, the averaged curve seems to capture the overall behavior well.
While there is variation between different realizations of precision matrices, for large values of $\log(1/\eps)$, the FTT rank trajectory in log-log space seems to only differ by a constant.
This indicates that the growth behavior of the FTT ranks is primarily dictated by the ranks in the subdiagonal blocks.
\begin{figure}[t]
    \center
    \includegraphics[width=0.95\linewidth]{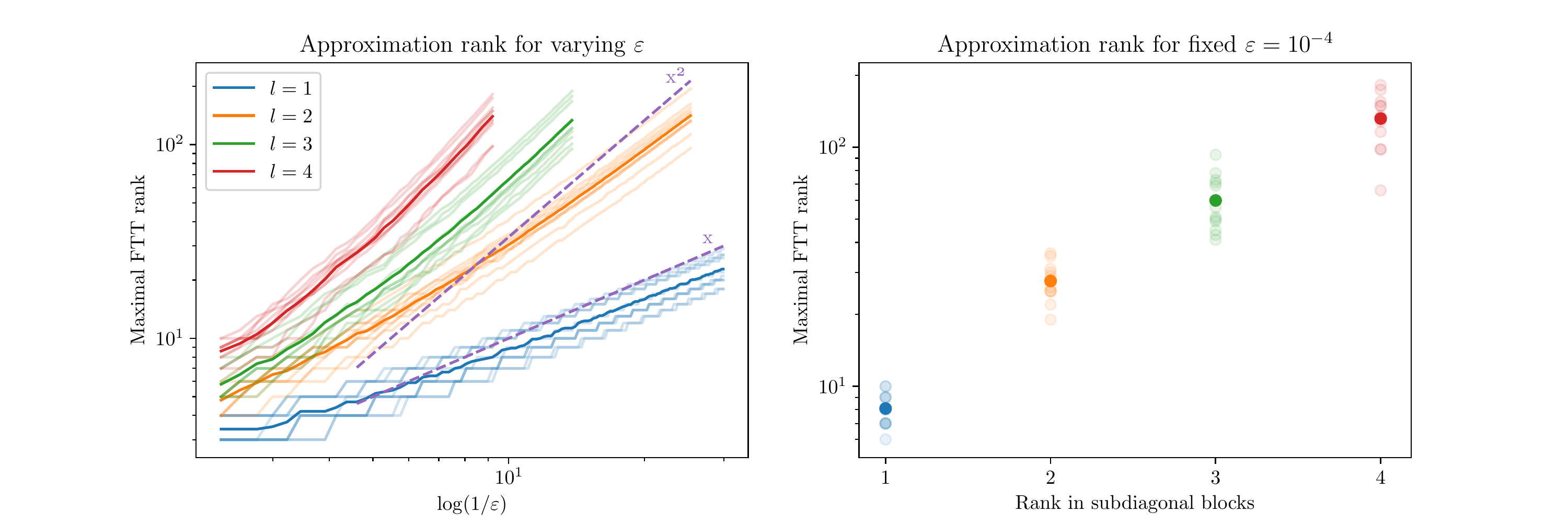}
    \caption{Ranks of the FTT approximations of a $15$ dimensional Gaussian density with low-rank subdiagonal blocks. 
    On the left, we show a log-log plot of the maximal FTT ranks versus
$1/\eps$, for different subdiagonal block ranks $l$.
    The dotted lines indicate the slope of linear and quadratic growth in $\log(1/\eps)$.
    On the right, we show the growth of the FTT ranks for a fixed relative accuracy $\eps = 10^{-4}$ when increasing the ranks in the subdiagonal blocks.}
    \label{fig:test-low-rank}
\end{figure}

After an initial phase, the FTT rank plot appears to be linear in the log-log plane which indicates that the predicted poly-logarithmic growth rate is qualitatively correct.
However, the rates in which the rank growths appears to be less than the rate $\log(1/\eps)^l$ suggested by Theorem \ref{thm:lowrankapprox}.
Nevertheless, when we compare the FTT ranks for for a fixed accuracy approximation of $\eps = 10^{-4}$ and vary the rank of the subdiagonal block, we can clearly see that on average, the necessary FTT rank increases exponentially in the number of non-zero singular values.
Therefore, our predicted exponential dependence on the number of singular values does not appear to be overly pessimistic.

\subsection{Exponentially Decaying Singular Values}\label{sec:numerics-exp-decay}

As a second test, we follow the setup of Section \ref{sec:numerics-low-rank} but fix the singular spectrum of each subdiagonal block to follow the sequence
\[
    \sigma_i = e^{-\theta i}
\]
for the $i$-th singular value in each matricization $k$ and vary $\theta \in \{2.5, 2, 1.5, 1\}$.
We increase the dimension to $d=30$ to make sure that we have sufficiently large subdiagonal blocks to not cut off relevant values of the singular spectrum.
The number of interpolation nodes $n$ varies between $230$ and $310$.

\begin{figure}
    \center
    \includegraphics[width=0.95\linewidth]{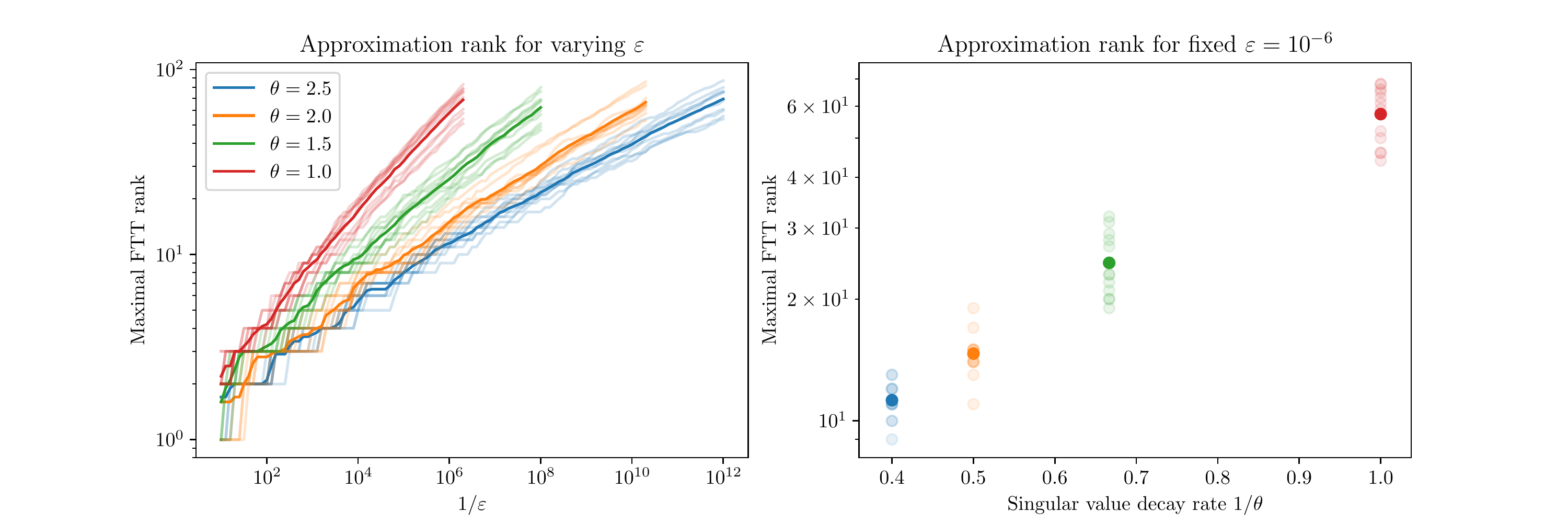}
    \caption{Ranks of the FTT approximation of a $30$ dimensional Gaussian density with exponentially decaying singular values in the subdiagonal blocks.
    On the left, we show a log-log plot of the maximal FTT ranks given the accuracy $1/\eps$ for different choices of the decay rate $\theta$ of the singular values in the subdiagonal blocks.
    On the right, we show the growth of the FTT ranks for a fixed relative accuracy $\eps = 10^{-6}$ when decreasing the decay rate $\theta$.}
    \label{fig:test-exp-decay}
\end{figure}
The result can be seen in Figure \ref{fig:test-exp-decay}.
Note that compared to Figure \ref{fig:test-low-rank}, the $x$-axis now shows $1/\eps$ instead of $\log(1/\eps)$.
Again, the FTT rank trajectories are to a high degree determined by the prescribed singular spectrum.
For $1/\eps$ sufficiently big, the rank growth appears to be polynomial in $1/\eps$.
This confirms the polynomial growth rate of FTT ranks predicted in Theorem \ref{thm:expdecayapprox}.
In particular, the ranks grow faster in the accuracy than for precision matrices with low-rank subdiagonal blocks.
Overall, as in the previous case, the numerical tests agree qualitatively with the predictions made by Theorem \ref{thm:expdecayapprox}.

\section{Application to Bayesian Filtering} 
\label{sec:filtering_example}

In this section, we present an example of how our theoretical results
from Section~\ref{sec:rankbounds} can be used to predict the 
TT ranks of numerical approximations of the filtering distribution in
Bayesian data assimilation of high-dimensional dynamical systems. The
state of a dynamical system $x(t)\in\R^d$ is assumed to evolve 
according to a stochastic differential equation
\[
\text{d}x = f(x) \text{d}t + \text{d}\xi_t\,,
\]
where $f(x)\in\R^d$ is a known force field, and $\xi_t$ is a random
noise process, such as a Wiener process~\cite{LasotaMackey}, with a
distribution possessing a density $\pi_{\xi}(\xi)$ that is independent of time.
The system is observed at discrete times $t_{\ell}$\,,
providing noisy measurements 
\[
	z_{\ell} = h(x(t_{\ell})) + \eta_{\ell}\,, \quad  \ell=1,2,\ldots,
\]
of the state,  where $h(x)$ is a known observation function and
$\eta_{\ell}$ is the observational noise. The distribution of the
observed  values $z_{\ell}$ given the state $x(t_{\ell})$ with density
$\rho(z_{\ell}|x(t_{\ell}))$  (also called the \emph{likelihood}) is fully 
described by the distribution of the noise vector $\eta_{\ell}$,
assumed to be given. All measurements up to a time~$t$ are collected
into a  set $Z(t) = \{z_{\ell}:~t_{\ell}\le t\}$. We will also use
$Z_k = Z(t_k)$ as shorthand notation.

Sequential Bayesian inference aims to find a time-dependent 
sequence of \emph{filtering densities} $\rho(x(t) | Z(t))$.
This process can be split into two steps~\cite{J70}:
\begin{enumerate}
 \item \emph{Prediction}: Since $Z(t)=Z_{\ell}$ is constant for $t \in
   [t_{\ell}, t_{\ell+1})$,  the continuous-time evolution of the
   filtering  density from $t_\ell$ to $t_{\ell+1}$ satisfies the 
   Chapman-Kolmogorov equation
 \[
  \rho(x(t_{\ell+1}) | Z_{\ell}) = \int \pi_{\xi}\left(x(t_{\ell+1}) -
    X(y;\Delta t_{\ell+1}) \right) \rho(y|Z_{\ell}) \, \text{d}y,
 \]
 where $X(y;\tau)$ denotes the deterministic solution of $x'
   = f(x)$ at time $\tau$ starting from the initial state $x(0)=y$
 and where $\Delta t_{\ell+1} = t_{\ell+1}-t_{\ell}$.
 In particular, if $\xi \sim \mathcal{N}(0,\epsilon I)$ (a vector of
 independent Brownian motions), the \emph{prediction density}
$\rho(x(t) | Z_{\ell})$, for $t \in
   [t_{\ell}, t_{\ell+1}]$, can be computed as the solution of the
 Fokker-Planck equation (FPE)
 \begin{equation}\label{eq:fpe}
  \frac{\partial \tilde{\rho}(x, \tau)}{\partial \tau} = \nabla_x\cdot (f
  \tilde{\rho}) + \epsilon \nabla_x^2 \tilde{\rho},
 \end{equation}
 starting from the initial state $\tilde{\rho}(x, 0) =
 \rho(x(t_{\ell}) | Z_{\ell})$ and integrating \eqref{eq:fpe} from
 $\tau=0$ to $\Delta t_{\ell+1}$.
 \item \emph{Update}: At time $t_{\ell+1}$, a new measurement
   $z_{\ell+1}$ can be assimilated using Bayes' rule to
   compute the new filtering distribution
 \begin{equation}\label{eq:filterup}
  \rho(x(t_{\ell+1}) | Z_{\ell+1}) = \frac{\rho(z_{\ell+1} | x(t_{\ell+1})) \rho(x(t_{\ell+1}) | Z_{\ell})}{\rho(z_{\ell+1} | Z_{\ell})},
 \end{equation}
where $\rho(x(t_{\ell+1}) | Z_{\ell}) = \tilde{\rho}(x(t_{\ell+1})
  , \Delta t_{\ell+1})$, i.e.\ the solution of \eqref{eq:fpe} at
  time $\tau= \Delta t_{\ell+1}$.
\end{enumerate}

Low-rank tensor formats, such as the TT or the FTT format, have
been used extensively to design efficient algorithms for the
Fokker-Planck equation \eqref{eq:fpe}
(cf.~\cite{DKhOs-parabolic1-2012,Sueli-greedy_FP-2012,
dk-qtt-tucker-2013,dektor-dyntensor-2020}),
as well as for approximating pointwise products and integrals of
multivariate functions as in \eqref{eq:filterup}
(cf.~\cite{ro-crossconv-2015}). 
However, results on the theoretical convergence analysis of the FTT
format applied to filtering densities are still lacking.

In the case of linear dynamics $f(x) = Fx$, a linear
observation function $h(x) = Hx$ and normally distributed $\xi \sim
\mathcal{N}(0,\epsilon I)$ and 
$\eta_{\ell} \sim \mathcal{N}(0,R)$,
the filtering distribution is also Gaussian with mean $\bar x(t)$
  and covariance matrix $C(t)$ defined by the \emph{Kalman filter}
  (which also consists of two steps):
\begin{enumerate}
 \item \emph{Prediction}: Solve the system of ODEs
 \begin{align}
  \frac{\text{d}\bar y}{\text{d}\tau} & = f(\bar y) \quad \text{and} \quad 
  \frac{\text{d} P}{\text{d}\tau} = F P + P F^\top + \epsilon I,
\label{eq:kalpred}
 \end{align}
starting from $\bar y(0) = \bar x(t_{\ell})$ and $P(0) = C(t_{\ell})$
and integrating from $\tau=0$ to $\Delta t_{\ell+1}$. 
 \item \emph{(Kalman) Update}: The system \eqref{eq:kalpred} can be
   solved explicitly, leading to
 \begin{align}
 \bar x(t_{\ell+1}) & = \bar y(\Delta t_{\ell+1}) + K \big(z_{\ell+1} - H
                      \bar y(\Delta t_{\ell+1})\big), \nonumber\\
  C(t_{\ell+1}) & = \big(I - K H\big) P(\Delta t_{\ell+1}) \label{eq:kalup}
 \end{align}
where  $K  = P(\Delta t_{\ell+1}) H^\top \big(H P(\Delta t_{\ell+1})
H^\top + R\big)^{-1}$ is the so-called \emph{Kalman gain}.
\end{enumerate}

For nonlinear dynamical systems, the Kalman filter can be applied to
the linearised system at time~$t_\ell$. This is the so-called
\emph{Extended Kalman filter (EKF)}, which is the de facto standard in
nonlinear state estimation due to its computational efficiency,
especially when compared to solving the Fokker-Planck equation
\eqref{eq:fpe}. However, in a general nonlinear system the EKF filter 
can rapidly diverge. Nevertheless, since it is relatively cheap to 
compute, the EKF covariance matrix $C(t_{\ell+1})$ can be used to
estimate the TT ranks of the corresponding Gaussian density function
via Theorem~\ref{thm:lowrankapprox} or \ref{thm:expdecayapprox}.
For systems that are \emph{weakly} nonlinear, relative to the frequency
of observations, the solution to \eqref{eq:fpe}--\eqref{eq:filterup} will exhibit
comparable TT ranks. One can then use the cheap information provided
by the EKF to initialise the ranks in an efficient TT or FTT algorithm
for the full Bayesian filtering problem
\eqref{eq:fpe}--\eqref{eq:filterup} and analyze its performance
rigorously. However, this goes beyond the scope of this paper and will
be the focus of a separate paper.

\subsection{Numerical example: coupled
  pendulums} \label{sec:filtering_example_numerics}

We illustrate the approach on a system of $N$ perturbed, weakly coupled
pendulums. The $k$th pendulum is described by its displacement from
the equilibrium $\theta_k$ and its velocity $\omega_k$. The
displacement follows the usual inertial law $d\theta_k/dt=\omega_k$,
and the acceleration is composed of the retracting gravity force,
and the coupling forces between adjacent pendulums. 
Overall, $N$ weakly coupled pendulums can be described by the
$d=2N$-dimensional kinematic evolution $x' = f(x)$ with
\begin{equation}
\label{eq:pend}
 x = \left(\begin{array}{c}
                           \theta_1 \\ \omega_1 \\
                           \theta_2 \\ \omega_2 \\
                           \vdots \\
                           \theta_N \\ \omega_N
                          \end{array}
                    \right) \quad \text{and} \quad
                  f(x) = \left(\begin{array}{cccc}
                           \omega_1 \\ -\sin\theta_1 +\kappa\left[\theta_2 - \theta_1\right] \\
                           \omega_2 \\ -\sin\theta_2 +\kappa\left[
                            (\theta_3 - \theta_2) - (\theta_2 - \theta_1) \right] \\
                           \vdots \\
                           \omega_N \\ -\sin\theta_N +\kappa\left[-(\theta_N - \theta_{N-1})\right]
                          \end{array}
                     \right).
 \end{equation}
Here $\kappa$ is a coupling coefficient that we set to $0.2$ in this
experiment. Note that the coupling force is similar to the discrete
one-dimensional Laplacian acting on the displacement variables. 
The filtering is
performed assuming that \eqref{eq:pend} is perturbed
by a vector $\xi_t$ of independent Brownian motions with variance
$10^{-3}$, modelling for example outside influences on the pendulums
or computational bias due to numerical treatment. However, an
important aspect is also that the perturbation stabilizes the filter.
The initial density $\rho(x|Z_0)$ is chosen to be
zero-mean Gaussian with covariance $C(0) = 0.09I$.

We synthesize observational data by an accurate noise-free integration
of the 'true' dynamics with a relative tolerance of $10^{-6}$,
starting from $\theta_k=0.25$ and $\omega_k=0$ at $t=0$. The
measurements are then produced as the positions of the first pendulum,
perturbed by Gaussian noise with variance $0.04$ for $t_\ell = 0.4 \cdot \ell$, $\ell=0,\ldots,250$. This gives the final time
$T = t_{250}=100$.

The exact (Bayes-optimal) filtering densities
  $\rho(x(t) | Z(t))$ for $0 < t \le T$, are now approximated using TT
  methods. In each filtering step
  from $t_\ell$ to $t_{\ell+1}$, we solve the Fokker-Planck
  equation~\eqref{eq:fpe} using tAMEn~\cite{d-tamen-2018} with the 
TT-SVD approximation threshold $\delta_{\text{appr}} = 10^{-2}$ (as
introduced in Section~\ref{sec:computingftt}). The same threshold is
used for approximating the pointwise product of TT decompositions in
the update  step~\eqref{eq:filterup}. At the same time, as outlined
above, we compute the Kalman update $C(t_{\ell+1})$ at $t_{\ell+1}$
for the linearised dynamics from \eqref{eq:kalup}, i.e., using the linearised version 
of~\eqref{eq:pend} at $t_\ell$ where all $\sin$ functions are replaced by their arguments.

\definecolor{C0}{RGB}{31,119,180}
\definecolor{C1}{RGB}{255,127,14}
\definecolor{C2}{RGB}{44,160,44}
\definecolor{C3}{RGB}{214,39,40}
\definecolor{C4}{RGB}{148,103,189}
\definecolor{C5}{RGB}{140,86,75}
\definecolor{C6}{RGB}{227,119,194}
\definecolor{C7}{RGB}{127,127,127}
\definecolor{C8}{RGB}{188,189,34}
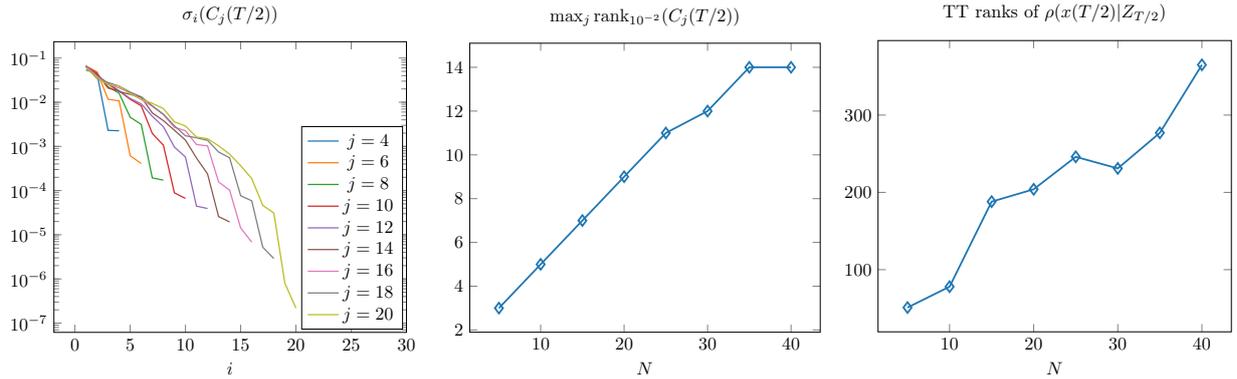
\begin{figure}[t]
\centering
\noindent\resizebox{0.35\linewidth}{!}{%
\begin{tikzpicture}
\begin{axis}[%
xmode = normal,
ymode = log,
xlabel = $i$,
title = $\sigma_i(C_j(T/2))$,
legend style={at={(0.99,0.01)},anchor=south east},
xmax=30,
]
\addplot[C0,line width=0.7pt,no marks] table[header=false,x index=0,y index=2]{Images/offsvd-N20-T100.dat}; \addlegendentry{$j=4$};
\addplot[C1,line width=0.7pt,no marks,width=1pt] table[header=false,x index=0,y index=3]{Images/offsvd-N20-T100.dat}; \addlegendentry{$j=6$};
\addplot[C2,line width=0.7pt,no marks] table[header=false,x index=0,y index=4]{Images/offsvd-N20-T100.dat}; \addlegendentry{$j=8$};
\addplot[C3,line width=0.7pt,no marks] table[header=false,x index=0,y index=5]{Images/offsvd-N20-T100.dat}; \addlegendentry{$j=10$};
\addplot[C4,line width=0.7pt,no marks] table[header=false,x index=0,y index=6]{Images/offsvd-N20-T100.dat}; \addlegendentry{$j=12$};
\addplot[C5,line width=0.7pt,no marks] table[header=false,x index=0,y index=7]{Images/offsvd-N20-T100.dat}; \addlegendentry{$j=14$};
\addplot[C6,line width=0.7pt,no marks] table[header=false,x index=0,y index=8]{Images/offsvd-N20-T100.dat}; \addlegendentry{$j=16$};
\addplot[C7,line width=0.7pt,no marks] table[header=false,x index=0,y index=9]{Images/offsvd-N20-T100.dat}; \addlegendentry{$j=18$};
\addplot[C8,line width=0.7pt,no marks] table[header=false,x index=0,y index=10]{Images/offsvd-N20-T100.dat}; \addlegendentry{$j=20$};
\end{axis}
\end{tikzpicture}
}~
\noindent\resizebox{0.32\linewidth}{!}{%
\begin{tikzpicture}
\begin{axis}[%
xmode = normal,
ymode = normal,
xlabel = $N$,
title = $\max_j\mathrm{rank}_{10^{-2}}(C_j(T/2))$,
legend style={at={(0.01,0.99)},anchor=north west},
]
\addplot[C0,line width=1pt,mark=diamond,mark size=3pt] table[header=true,x=N,y=rmax]{Images/offranks-tol1e-2.dat}; 
\end{axis}
\end{tikzpicture}
}~
\noindent\resizebox{0.33\linewidth}{!}{%
\begin{tikzpicture}
\begin{axis}[%
xmode = normal,
ymode = normal,
xlabel = $N$,
title = TT ranks of $\rho(x(T/2)|Z_{T/2})$,
legend style={at={(0.01,0.99)},anchor=north west},
]
\addplot[C0,line width=1pt,mark=diamond,mark size=3pt] table[header=true,x=N,y=rmax]{Images/Npendulum_new_ranks.dat}; 
\end{axis}
\end{tikzpicture}
}
\caption{Results of the filtering example. In the left figure, we plot the singular values of the
  off-diagonal blocks $C_j(T/2)$ of the covariance matrix $C(T/2)$ 
  computed by the EKF at time $T/2$. The maximum numerical ranks of
    $C_j(T/2)$, truncated at $10^{-2}$, and the TT ranks of the
    approximate solutions $\rho(x(T/2)|Z_{T/2})$
    of the exact Bayes filtering problem
    \eqref{eq:fpe}--\eqref{eq:filterup} at time $T/2$ are plotted as
    functions of $N$ in the middle and right panels, respectively.}
\label{fig:pend}
\end{figure}

In Figure~\ref{fig:pend} (left), we track the singular values of the
off-diagonal blocks $C_j(T/2) := C_{j+1:d,\,1:j}(T/2)$
of the covariance matrix of the extended Kalman filter (using MATLAB
notation to denote the subblocks) for $N=20$ and $\ell=125$,
i.e.\ at time $t_{\ell}=T/2$. In the beginning of the filtering, the density
  function is close to the initial product density, while near the final time $T$
  the distribution follows the state of the deterministic system
  with a small and isotropic uncertainty around it. At intermediate
  times the highest correlations between the variables are observed,
  and consequently the ranks of $C_j$ and of the TT approximation of the solution 
  to the Fokker-Planck equation are largest.

We vary the number of pendulums $N$ (and consequently the dimension $d=2N$),
and plot the maximum numerical ranks of $C_j(T/2)$ over all
  possible off-diagonal blocks, truncated at $10^{-2}$ (middle
  figure), as well  as the TT ranks of the approximate nonlinear
  filtering density $\rho(x(T/2)|Z(T/2))$ (right figure). 
The singular values exhibit an exponential decay, which is probably
due to the local coupling of the components of the system.
Moreover, both the ranks of the off-diagonal blocks of the covariance
matrix, and the TT ranks of the Bayes-optimal filtering solution grow
linearly with the number of pendulums. This is in line with the
results of Theorem~\ref{thm:expdecayapprox} where we predicted an
algebraic scaling of the TT ranks with the dimension. 
A more detailed analysis of this application, in particular
  comparing with alternative approaches and 
  highlighting its efficiency and its excellent scaling with dimension,
  are left for a separate publication.

\section{Conclusions} 
\label{sec:conclusions}

The present paper introduces rigorous a-priori bounds on the necessary ranks to represent Gaussian densities to a specified accuracy in the functional Tensor-Train (FTT) format. We showed that the FTT ranks can be related to the singular spectrum of the subdiagonal blocks of the precision matrix of the Gaussian distribution. In particular, under certain  assumptions on the correlation structure, the Tensor-Train format was shown to provide an efficient surrogate for high dimensional Gaussian distributions that is not affected by the curse of dimensionality. As motivated carefully in the introduction, we see this as a stepping stone for a rigorous analysis of low-rank tensor approximation of more general distributions. The rank bounds that we proved are not sharp, but as the numerical results in Section \ref{sec:numerics} show, we managed to qualitatively capture the decay rate and the dependence on the required accuracy~$\eps$.
The dependence on the number of dimensions appears to be overly pessimistic, but this dependence is in part due to fundamental properties of the Tensor-Train format and it is not clear how this can be avoided.

By utilizing properties of the (functional) Tensor-Train format, the results can immediately be generalized, e.g. to sums and (elementwise) products of Gaussian densities. Therefore, our theory also covers cases like Gaussian mixture models, where each element of the mixture fulfills one of our precision matrix conditions individually.
Nevertheless, it would be of great interest to establish rank bounds also for more general classes of probability distributions (e.g.\ distributions of the exponential family).

\section*{Acknowledgments}
The first and third author gratefully acknowledge support by the DFG priority programme 1648 (SPPEXA) under Grant No.\ GR-3179/4-1 and GR-3179/4-2.
The first author gratefully acknowledges support by the Leverhulme Trust through project Grant No.\ RPG-2017-203 and by a scholarship of the Education Fund of the RWTH Aachen University, supported by a donation from TÜV Rheinland Stiftung.

\bibliographystyle{plain}
\bibliography{paper} 

\end{document}